\documentclass[12pt]{amsart}

\usepackage[UKenglish]{babel} 
\usepackage[T1]{fontenc}
\usepackage[ansinew]{inputenc}
\usepackage{graphicx} 
\usepackage{lmodern} 
\usepackage{enumerate}
\usepackage{amsmath}
\usepackage{amsthm}
\usepackage{amsfonts}
\usepackage{amssymb}
\usepackage{amscd}
\usepackage{xcolor}
\usepackage{cleveref}
\usepackage[numbers,sort]{natbib}
\usepackage[colorinlistoftodos]{todonotes}
\usepackage{mathtools}

\newcommand{\newword}[1]{\emph{#1}}
\newcommand{\ldb}{\{\!\!\{}
\newcommand{\rdb}{\}\!\!\}}
\newcommand{\PP}{\mathbb{P}} 
\newcommand{\AAA}{\mathcal{A}}

\newcommand{\BB}{\mathcal{B}}
\newcommand{\CF}{\mathcal{CF}}
\newcommand{\DF}{\mathcal{DF}}

\newcommand{\FF}{\mathcal{F}}

\newcommand{\MM}{\mathcal{M}}

\newcommand{\RR}{\mathbb{R}}
\newcommand{\TT}{\mathbb{T}}

\newcommand{\cork}{\operatorname{cork}}
\newcommand{\relint}{\operatorname{relint}}
\newcommand{\rk}{\operatorname{rk}}
\newcommand{\cl}{\operatorname{cl}}
\newcommand{\stiefel}{\pi} 
\newcommand{\relsupp}{\operatorname{rs}} 
\newcommand{\dapx}{\mathcal{DA}} 
\newcommand{\dmat}{\mathcal{DM}} 
\newcommand{\cocl}{\operatorname{cocl}} 
\newcommand{\TP}{\mathbb{TP}^{n-1}} 
\newcommand{\pl}[2]{#1_{#2}} 
\newcommand{\f}[2]{#1_{#2}} 
\newcommand{\vertex}[2]{v^{#1}_{#2}} 
\newcommand{\Be}[1]{\mathcal{L}(#1)} 
\newcommand{\underm}[1]{\mathchoice%
{\underbracket[1pt][-1pt]{\displaystyle#1}}%
{\underbracket[1pt][-1pt]{\textstyle#1}}%
{\underbracket[1pt][-1pt]{\scriptstyle#1}}%
{\underbracket[1pt][-1pt]{\scriptscriptstyle#1}}} 
\newcommand{\present}{\mathcal{A}} 
\newcommand{\stableint}{\mathbin{\!\mathop\cap\limits_{\mathrm{stable}}\!}} 

\theoremstyle{plain}

\newtheorem{proposition}{Proposition}[section]
\newtheorem{theorem}[proposition]{Theorem}
\newtheorem{lemma}[proposition]{Lemma}
\newtheorem{corollary}[proposition]{Corollary}
\newtheorem{introTheorem}[proposition]{Theorem}

\theoremstyle{definition}
\newtheorem{definition}[proposition]{Definition}
\newtheorem{example}[proposition]{Example}

\theoremstyle{remark}
\newtheorem{remark}[proposition]{Remark}

\numberwithin{equation}{section}

\renewcommand{\labelenumi}{\textup{(\arabic{enumi})}}

\begin{document}

\pagestyle{plain}

\title{Presentations of Transversal Valuated Matroids}
\author{Alex Fink\and Jorge Alberto Olarte}
\date{}

\begin{abstract}
Given $d$ row vectors of $n$ tropical numbers, $d<n$,
the tropical Stiefel map constructs a version of their row space,
whose Pl\"ucker coordinates are tropical determinants.
We explicitly describe the fibers of this map.
From the viewpoint of matroid theory, 
the tropical Stiefel map defines a generalization of transversal matroids in the valuated context,
and our results are the valuated generalizations of theorems of Brualdi and Dinolt, Mason and others
on the set of all set families that present a given transversal matroid.
We show that a connected valuated matroid is transversal if and only if all of its connected initial matroids are.
The duals of our results describe complete stable intersections of tropical linear spaces via valuated strict gammoids.
\end{abstract}

\maketitle

\section{Introduction}\label{sec:intro}
In tropical mathematics, the accepted definition of \newword{tropical linear spaces}
uses an analogue to vectors of Pl\"ucker coordinates. 
These vectors were introduced by Dress and Wenzel \cite{DressWenzel},
who named them \newword{valuated matroids} because matroids appear as a special case.

Over a field $\mathbb K$, every linear subspace of $\mathbb K^n$ can also be described as the rowspace of some matrix with entries in $\mathbb K$.
The tropical counterpart fails.
The \newword{tropical Stiefel map} $\stiefel$ of~\cite{FR} sends a matrix of tropical numbers to the tropical linear space determined by its vector of maximal minors;
however, not all tropical linear spaces arise in this way.


The combinatorics of the map $\stiefel$ is governed by \newword{transversal matroids}.
Let $\present=\ldb A_1,\ldots,A_d\rdb$ be a multiset of subsets of a finite set~$E$.
Edmonds and Fulkerson \cite{EdmondsFulkerson} observed that the set of subsets $J\subseteq E$
which form a \newword{transversal} of~$\present$,
i.e.\ such that there is an injection $f:J\to\{1,\ldots,d\}$ with $j\in A_{f(j)}$ for each $j\in J$,
are the independent sets of a matroid. 
A matroid $M$ arising in this way is called a transversal matroid,
and $\present$ is called a \newword{presentation} of~$M$.
To emphasize the commonality between valuated and unvaluated cases,
we define a \newword{transversal valuated matroid} $V$ to be a valuated matroid in the image of~$\stiefel$,
i.e.\ a vector of tropical maximal minors of a $d\times n$ matrix $A$ of tropical numbers.
The matroids that are transversal valuated matroids are exactly the transversal matroids.

Brualdi and Dinolt described all presentations of a given transversal matroid.
(Their original formulation \cite[Theorem 5.2.6]{BrualdiDinolt} is \Cref{BrualdiDinolt} below.)
Any transversal matroid $M$ has a unique maximal presentation, which consists of $\tau_M(F)$ copies of~$E\setminus F$
for each flat $F$ of~$M$,
where the number $\tau_M(F)$ is computed by a recurrence \eqref{eq:definition of tau} on the lattice of flats.
Every presentation $\ldb E\setminus F_1,\ldots,E\setminus F_d\rdb$ of~$M$ can be obtained from the maximal one by deleting relative coloops
in a way that doesn't contravene Hall's theorem, i.e.\ that satisfies
\begin{equation}\label{eq:Hall intro}
\cork(\bigcap_{i\in I} F_i) \le |I| 
\end{equation}
for every $I\subseteq\{1,\ldots,d\}$, where $\cork(J)=d-\rk(J)$ is the corank function.

Our main theorem is an explicit description of the fibers of~$\stiefel$.
\begin{introTheorem}[Synopsis of \Cref{thm_B}]\label{intro_thm_B}
Each nonempty fiber of the tropical Stiefel map $\stiefel$ is the orbit of a fan in the space of $d\times n$ tropical matrices under the action of~$S_d$ permuting the rows.
\end{introTheorem}
This directly generalizes Brualdi and Dinolt's result to valuated matroids.
For (unvaluated) matroids in the image of~$\stiefel$,
the apex of our fan is the unique maximal presentation of Brualdi and Dinolt.
Apart from a lineality space, 
all rays of our fan are in coordinate directions, 
and the sets of coordinates that appear are described by a ``local'' reformulation of equation~\eqref{eq:Hall intro}.

In \cite{FR} a necessary condition for a valuated matroid $V$ to be transversal was given (\Cref{prop:FR local}).
Assuming for convenience that $V$ is connected,
the condition is that if $V$ is transversal, all connected initial matroids of~$V$ must be transversal.
The initial matroids are those whose matroid polytopes appear in the polytope subdivision induced by~$V$.
We obtain a converse.
\begin{introTheorem}[= \Cref{coro_transv}]
A connected valuated matroid is transversal if and only if all of its connected initial matroids are transversal.
\end{introTheorem}

Duality of valuated matroids
replaces the tropical Stiefel map by the process of taking the \newword{stable intersection} of a collection of tropical hyperplanes.
In the realm of matroids, the dual of the class of transversal matroids is the class of \newword{strict gammoids}.
This class arises from flows in directed graphs, which admit a natural generalization to the realm of valuated matroids which we call \newword{valuated strict gammoids}.
We find the statements derived from \Cref{coro_transv} by this duality to be of interest in their own right.
\begin{introTheorem}[= \Cref{thm:gammoids}]
Let $V$ be a valuated matroid and $L$ its corresponding tropical linear space. Then the following are equivalent:
\begin{enumerate}
	\item $L$ is the stable intersection of tropical hyperplanes.
	\item $V$ is a valuated strict gammoid.
	\item Near each point, $L$ is locally the Bergman fan of a strict gammoid.
\end{enumerate}
Furthermore, \Cref{thm_B} explicitly describes the spaces of all $d$-tuples of tropical hyperplanes whose stable intersection is a given tropical linear space,
and of all weighted directed graphs that present a given valuated strict gammoid.
\end{introTheorem}

In this paper, \Cref{sec:background} reviews valuated matroids and tropical linear spaces.
\Cref{sec:transversality} introduces transversality and the Stiefel map, and interprets the former as the $\{0,\infty\}$-valued case of the latter.
We begin to characterize presentations in \Cref{sec:regions},
by bounds on the number of rows chosen from certain regions of the tropical linear space.
\Cref{sec:matroid valuations} introduces a piece of technical apparatus needed for the proofs of the main theorems,
after which \Cref{sec:presentations} proves them.
\Cref{sec:stable} introduces strict gammoids and stable intersection and reframes our results in this language.

\subsection*{Acknowledgments}
During this work the first author received support from
the Deutsche For\-schungs\-ge\-mein\-schaft project ``Facetten der Komplexit\"at''
and from the European Union's Horizon 2020 research and innovation programme under grant agreement No~792432.
The first author also thanks the Mittag-Leffler Institute for their hospitality and delightful working conditions.
The second author was supported by the Einstein Foundation Berlin through the visiting fellowship of Francisco Santos.
We thank Michael Joswig, Georg Loho, and a referee for valuable feedback.

\section{Valuated matroids and tropical linear spaces}
\label{sec:background}
This section is a review of standard concepts to set up the terminology and notation;
it contains no new material.
Our work's main characters are tropical linear spaces,
or to give them another of their cryptomorphic names, valuated matroids \cite{DressWenzel}.
We recommend \cite[chap.~4]{MaclaganSturmfels} as a more detailed reference for tropical linear spaces and valuated matroids.
For (unvaluated) matroids, any standard textbook will suffice.

Fix a set $[n]=\{1,\ldots,n\}$.
We denote the set of all subsets of~$[n]$ with cardinality $d$ by $\binom{[n]}d$.
Given a subset $J\subseteq[n]$, we denote its zero-one indicator vector by
\[e_J = \sum_{j\in J}e_j\in\RR^n.\]

We distinguish multisets from sets by writing them with doubled braces, like $\ldb0,0,1\rdb$.

In the theory of valuated matroids, coordinates are drawn from
the semiring $\TT=\RR\cup\{\infty\}$ of \newword{tropical numbers}, with operations $\oplus:= \min$ and $\odot := +$ and identity elements $\infty$ and $0$.
The set $\TT^n$ of vectors of $n$ tropical numbers plays the role of affine $n$-space in tropical geometry.
But we prefer to work in projective space:
\[\TP = \big(\TT^n\setminus\{(\infty,\ldots,\infty)\}\big)\big/\RR(1,\ldots,1)\]
where the action of $\RR(1,\ldots,1)$ is by addition.
When we speak of the relative interior $\relint(P)$ of a polyhedron $P\subseteq\TP$,
we exclude the points which have more coordinates equal to $\infty$ than a generic point of~$P$ does,
i.e.\ the points on the ``faces at infinity'' of~$P$.

\subsection{Valuated matroids and matroid polytopes}\label{ssec:valuated}
A \newword{valuated matroid} $V$ on the ground set $[n]$,
whose \newword{rank} is an integer $\rk(V)=d$ with $0\leq d\leq n$,
is a vector in $\TT\PP^{\binom{n}{d} -1}$ whose coordinates are labeled by $\binom{[n]}d$
satisfying the tropical Pl\"ucker relations:
for any sets $A\in\binom{[n]}{d-1}$ and $C\in\binom{[n]}{d+1}$,
there is more than one index $j\in C\setminus A$ at which
$\pl V{A\cup\{j\}} + \pl V{C\setminus\{j\}}$
attains its minimal value.

Given a valuated matroid $V$, the set of all $B\in\binom{n}{d}$ such that $\pl VB$ is finite is the set of \emph{bases} of a \emph{matroid}, called the matroid \emph{underlying} $V$. 
Following the notation used in \cite{BakerBowler}, we write $\underm V$ for the matroid underlying $V$.
For a matroid $M$ we write $\BB(M)$ for the set of bases of $M$.
In this work we often look at matroids (cryptomorphically)
as the special case of valuated matroids that only have $0$ and $\infty$ coordinates:
that is, $\pl MB=0$ if $B\in\BB(M)$ and $\pl MB=\infty$ otherwise.

For a subset of $J\subseteq [n]$ we write $\rk_M(J)$ for the \newword{rank} of $J$ in $M$,
$\cl_M(J)$ for its \newword{closure},
$M|J$ for the \newword{restriction} of $M$ to $J$, 
$M/J$ for the \newword{contraction} of $J$ in $M$, and
$M\setminus J$ for the \newword{deletion} of $J$ in $M$.
We write $M^*$ for the \emph{dual} of $M$,
$\FF(M)$ for the \newword{lattice of flats} of $M$, 
and $\CF(M)$ for the \newword{lattice of cyclic flats},
i.e.\ $F\in\CF(M)$ if and only if $F\in \FF(M)$ and $[n]\setminus F\in \FF(M^*)$.
A \newword{cyclic set} of~$M$ is the complement of a flat of~$M^*$, equivalently a union of zero or more circuits of~$M$.
The \newword{coclosure} of $J\subseteq[n]$ is the largest cyclic set contained in $J$, 
in other words, $\cocl_M(J) := [n]\setminus\cl_{M^*}(J)$.
The \newword{corank} of $J$ is $\cork(J) = d-\rk(J)$.
We write $M_1\oplus M_2$ for the \newword{direct sum} of $M_1$ and $M_2$.


The \newword{matroid polytope} of~$M$ is 
\[P_M := \operatorname{conv}\{e_B:B\in\BB(M)\}\subseteq\RR^n.\]
The dimension of $P_M$ is equal to $d$ minus the number of connected components of $M$. 
For any $F\in \FF(M)$ the intersection of $P_M$ with the hyperplane $\left\{\sum\limits_{j\in F}x_j = \rk(F)\right\}$ is a face of $P_M$ and it is the polytope of the matroid $M|F\oplus M/F$. 
Any facet of $P_M$ which intersects the interior of $\Delta(d,n)$ is of this form for a cyclic flat $F\in \CF(M)$, and all the other facets are also of this form for some singleton $F$.


A valuated matroid $V$ with underlying matroid $M$ can be regarded as a height function on the vertices of the polytope $P_M$.
Such a height function produces a regular subdivision of~$P_M$ in the sense of \cite[Definition 2.2.10]{dLRS}. 
A real-valued function from the vertices of $P_M$ is a matroid subdivision if and only if all the faces of the induced regular subdivision are matroid polytopes \cite[Proposition 2.2]{Speyer}.
A vector $x\in \RR^n$ selects a face of the regular subdivision induced by~$V$ by taking the convex hull of all vertices $e_B$ of $P_B$ such that $V_B - \sum_{i\in B}x_i$ is minimized.
Such a face corresponds to the polytope of a matroid which we write $V^x$ known as the \newword{initial matroid} of $V$ at~$x$.
We write $\MM(V)$ for the set of all initial matroids of~$V$ all of whose loops are loops in~$V$. 

\begin{example}
\label{ex1}
Consider the uniform matroid $U_{2,4}$. Its matroid polytope is the hypersimplex $\Delta_{2,4}$ which is an octahedron. 
Now consider the valuated matroid $V$ where $V_{34}= 1$ and $V_B = 0$ for every $B\in {4\choose 2}\setminus \{34\}$. 
The matroid subdivision induced by $V$ divides the octahedron into two square pyramids, one with apex $e_{12}$ and the other one with apex $e_{34}$. The only $x$ that selects the pyramid with apex $e_{12}$ is $[0:0:0:0]$ while the only $x$ that selects the pyramid with apex $e_{34}$ is $[0:0:1:1]$. 
The initial matroids contained in $\MM(V)$ are those whose polytopes are
the two square pyramids, their common square face,
and four of the triangular faces, 
namely $\operatorname{conv}\{e_{12},e_{13},e_{14}\}$ and its $S_4$-images.
\end{example}

\subsection{Tropical linear spaces}

The (projective) \emph{tropical linear space} associated to a valuated matroid $V$ is
\begin{multline*}
\Be V := \{x=(x_1:\dots:x_n)\in\TP : \mbox{for any $C\in\binom{[n]}{d+1}$,}\\
\mbox{more than one $j\in C$ minimizes $x_j+\pl V{C\setminus\{j\}}$}.\}
\end{multline*}
We call $\Be V$ a \emph{tropical hyperplane} if $V$ has rank $d=n-1$.

We describe the polyhedral structure of a tropical linear space $L = \Be V$ using the language of matroids.
For simplicity, we assume throughout that $\underm V$ has no loops or coloops.
Define
\[
L^\circ := \{x\in\RR^n : 
\mbox{$V^x$ has no loops}\}.
\]
We have that $L$ is the closure of $L^\circ /\RR(1,\ldots,1)$ within $\TP$,
where the closure operation only adds points with infinite coordinates
(\cite[Prop~2.3]{Speyer}; implicit in \cite{Kapranov1992}).
The complex $L$ is pure of dimension~$d-1$.
The polyhedral complex structure of~$L$ is determined by the faces in~$L^\circ$:
the interiors of these faces are the sets of points $x\in\RR^n$
such that the matroid $V^x$ is constant.
For a matroid $M\in \MM(V)$, we write $\f LM$ for its corresponding cell, that is:
\[
\f LM := \overline{\iota_J(\{x\in L^\circ : (V|J)^x= M|J\})},
\]
where $J$ is the set of all nonloops of~$M$
and $\iota_J:\RR^J\to\TT^n$ is the inclusion filling in infinities in the missing coordinates.
When this cell is 0-dimensional, i.e. when $M$ is connected, we call it $\vertex LM$
(pedantically, $\vertex LM$ is the point which is the single element of~$\f LM$). 

\begin{example}
\label{ex1.2}
Consider the valuated matroid $V$ from~\Cref{ex1}. 
The polytopes in the subdivision induced by $V$ that correspond to loopless matroids are the two square pyramids, the square separating the pyramids and the four triangles which are inside each of the hyperplanes $x_i = 1$ for $i\in[4]$. \Cref{fig_ex1} shows a picture of the associated linear space.
\end{example}

\medskip

If $M$ is a matroid, the polyhedral complex structure we have just placed on the tropical linear space $\Be M$ 
is the Bergman fan as in \cite{FS05}, with the `coarse subdivision' as in \cite{AK06}.

We will use a construction of the set $\Be M$ in terms of flats throughout. 

\begin{proposition}[\cite{MaclaganSturmfels}, Theorem 4.2.6]\label{prop:fine fan}
Let $M$ be a matroid with no loops.  Then
\[\Be M^\circ = \left\{\lambda e_{[n]}+\sum_{i=1}^s a_{F_i}e_{F_i} :
\lambda\in\RR, a_{F_i}\geq0, F_1\subset\cdots\subset F_s\in\FF(M)\right\}.\]
\end{proposition}
The above shows that, as a set, the Bergman fan is the order complex of the lattice of flats, which endows the Bergman fan with its `fine subdivision' structure, 
also known as the nested set complex of~$M$.

If $L= \Be V$ is a tropical linear space and $x\in\RR^n/\RR(1,\ldots,1)$ is in the relative interior of~$\f LM$, then $\Be M$ equals the set of vectors $y$ such that $x+\varepsilon y\in L$ for all sufficiently small $\varepsilon>0$.
That is, $L$ looks like the translation $\Be M+x$ locally near~$x$.

\medskip

Valuated matroids have analogs of dual, restriction and contraction.
The \newword{dual} of $V$ is the valuated matroid $V^*$ of rank $n-d$ given by $V^*_B := V_{[n]\setminus B}$. Notice that $(V^*)^* = V$.
Let $J$ be an arbitrary subset of~$[n]$ and $B_{\rm c}$ any basis of $M/J$. 
Then the \newword{restriction} of $V$ to~$J$ is the valuated matroid $V|J$ on the ground set $J$ of~rank $k=d-|B_{\rm c}|$
such that $V|J_B = V_{B\cup B_{\rm c}}$ for any $B\in {J\choose k}$. 
This definition does not depend on the choice of $B_{\rm c}\in \BB(M/J)$, as choosing a different basis means tropically scaling all Pl\"ucker coordinates by the same factor.
In particular $\underm V| J = \underm{V|J}$. 
The \newword{contraction} of $J$ in~$V$ can be defined as $V/J := (V^*|([n]\setminus J))^*$.

Lemma~4.1.11 of~\cite{Frenk} describes the effects of deletion and contraction on $\Be{V}$.
Given a subset $A\subseteq [n]$ we have that 
\[
\Be{V/A} = \{x\in \TT\PP^{|[n]\setminus A|-1}: \hat x \in L\}
\] 
where $\hat x \in \TP$ is the extension of $x$ by setting the coordinates indexed by $A$ to be $\infty$. 
Let $\TP_A := \{x\in \TP : \exists i\in A \enspace x_i \ne \infty\}$ and let $\pi_A : \TP_A \to \TT\PP^{|A|-1}$ be the projection of $x$ to the coordinates indexed by $A$. 
Then
\[
\Be{V|A} = \pi_A(L\cap \TP_A).
\]  

\section{Transversality}\label{sec:transversality}
We recommend \cite{BrualdiWhite} as a general reference for transversal matroids.

\subsection{The tropical Stiefel map}
The fibers of the following map $\stiefel$ are our main subject.
\begin{definition}[\cite{FR}]
Let $A\in\TT^{d\times n}$ be a tropical matrix.
The \newword{tropical Stiefel map} is the partial function $\stiefel$
assigning to $A\in\TT^{d\times n}$ 
the valuated matroid $\pi(A)\in\TT\PP^{\binom{n}{d} -1}$ \cite[Example 5.2.3]{Murota} defined by
\[\pl{\stiefel(A)}B = \min\left\{\sum_{i=1}^d A_{i,j_i} : \{j_1,\ldots,j_d\}=B\right\}.\]
\end{definition}

The minimum on the right hand side of this equation,
over the $d!$ allocations of the names $j_1,\dots,j_d$ to the elements of~$B$,
is a \newword{tropical maximal minor} of~$A$.
The history of the connection between transversals and determinants goes back at least to~\cite{Edmonds1967}.

\begin{remark}\label{rem:domain}
The domain of~$\stiefel$ is the subset of $\TT^{d\times n}$ where at least one injective function $j:[d]\to[n]$
achieves $A_{i,j(i)}\neq\infty$ for all~$i\in[d]$.
By Hall's theorem, the only matrices excluded from the domain are those that have a $k\times(n+1-k)$ submatrix 
all of whose entries are~$\infty$
for some $1\leq k\leq d$.
\end{remark}

\begin{figure}[hbt]
	\centering
		\includegraphics[scale = 0.3125]{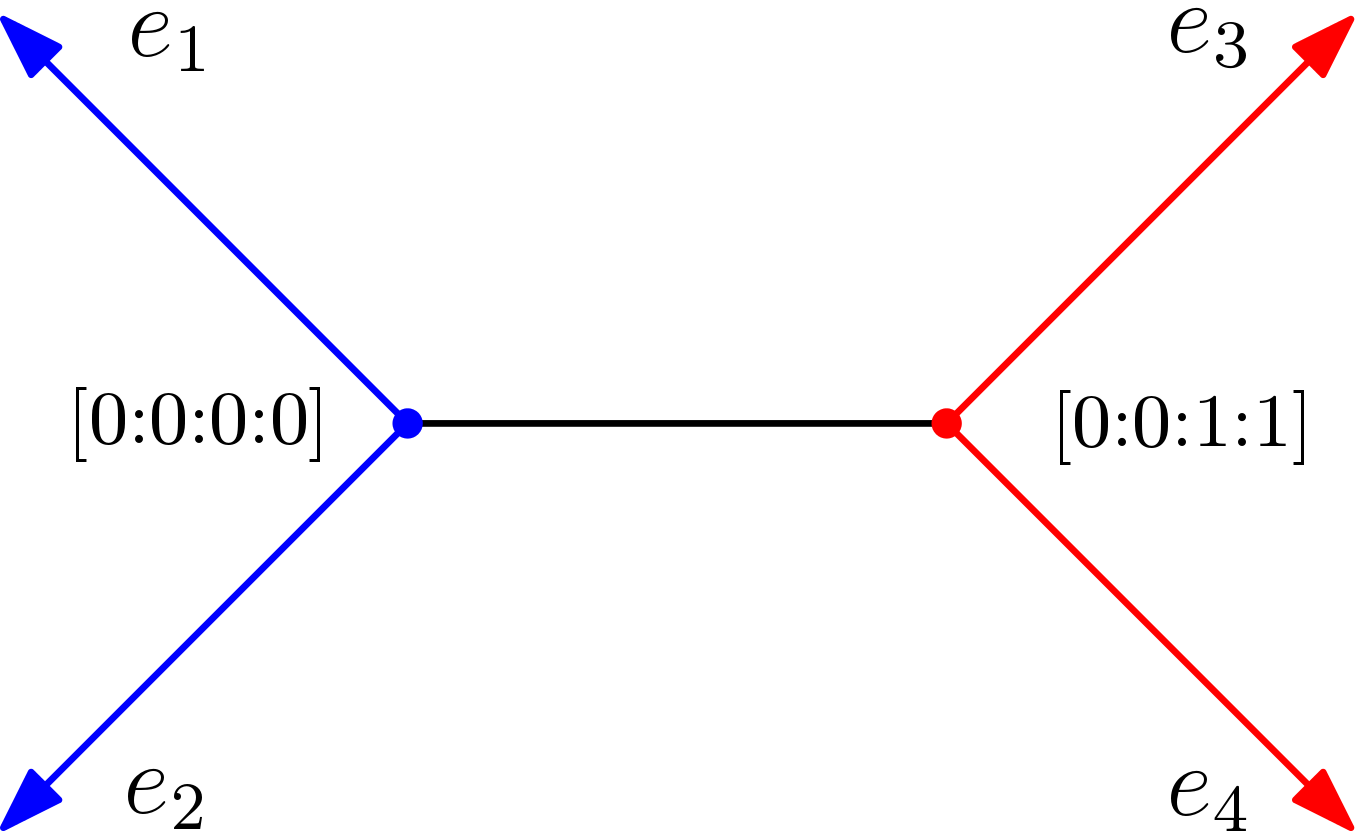}
	\caption{The tropical linear space $\Be{\stiefel(A)}\subseteq\mathbb{TP}^3$ of \Cref{ex:1}.}
	\label{fig_ex1}
\end{figure}

\begin{example}
\label{ex:1}
Consider the matrix
\[
A = 
\begin{pmatrix}
0 & 0 & 0 & 0\\
0 & 0 & 1 & 1
\end{pmatrix}
\]
in $\TT^{2\times 4}$.
Computing the tropical minors gives $\pl{\stiefel(A)}B = 0$ for any $B\in {[4]\choose 2}\setminus \{3,4\}$ and $\pl{\stiefel(A)}{34}= 1$, which is the same valuated matroid as in \Cref{ex1,ex1.2}. 
Notice that replacing either $A_{1,1}$ or $A_{1,2}$ (but not both at the same time) by any tropical number larger than~0 does not change any of the minors, so the resulting matrix would be mapped to the same valuated matroid.
Similarly, replacing either $A_{2,3}$ or $A_{2,4}$ by a number larger than~$1$ also does not change $\stiefel(A)$. 
\Cref{fig_ex1} shows the tropical linear space of~$\stiefel(A)$. 
Any matrix $A'$ with $\stiefel(A')= \stiefel(A)$ must have one row giving projective coordinates for a point in the blue subcomplex of the figure, 
and the other row doing the same for the red subcomplex.
Later, we will show how all fibers of~$\stiefel$ have a similar behavior.
\end{example}

Permuting the rows of $A$, or adding a scalar to any row, does not change $\pi(A)$,
and therefore neither does left multiplication by any invertible tropical matrix.
The first invariance implies that $\stiefel(A)$ is determined by the list of the projectivization (lying in $\TP$)
of each row of~$A$,
and the second invariance means that $\stiefel(A)$ is determined by the \emph{unordered} list, i.e.\ the multiset, of these projectivizations.
So we will normally discuss fibers of $\stiefel$ in terms of such multisets.
\begin{definition}\label{def:presentation}
A (\/\newword{transversal}\/) \newword{presentation} of a valuated matroid $V$ of rank~$d$
is a multiset $\present$ of $d$ points in $\TP$
such that $V=\stiefel(A)$, where $A$ is a matrix whose rows are coordinate vectors for the elements of~$\present$.
\end{definition}

If we say that a multiset $\present$ is a presentation of a tropical linear space $\Be V$,
we mean that it is a presentation of~$V$. 

The tropical Stiefel map is not surjective onto the space of valuated matroids. 
In \cite{FR} the name \emph{Stiefel tropical linear space} was given to tropical linear spaces of the form $\Be{\pi(A)}$.
We grant the valuated matroids another name motivated in what follows:
\begin{definition}\label{def:transversal}
A valuated matroid $V\in\TT\PP^{\binom{n}{d}-1}$ is \emph{transversal} if it is in the image of $\stiefel$.
An unvaluated matroid $M$ is \emph{transversal} if it is the underlying matroid of a transversal valuated matroid.
\end{definition}

Note that a transversal valuated matroid is not merely an arbitrary valuated matroid whose underlying matroid is transversal.
A counterexample is the valuated matroid $V$ of~\Cref{fig:snowflake},
whose underlying matroid is the transversal matroid $U_{2,6}$, 
but which is not transversal itself as explained in \Cref{ex:no transversal}.


Let us understand why \Cref{def:transversal} agrees with the classical definition of a transversal matroid.
Classically, a \newword{set system} presentation of a transversal matroid on $[n]$ is a multiset $\present$ of subsets of $[n]$.
A set is independent if there is a matching i.e. $J$ is independent if there is an injective function $\sigma:J \to \present$ such that $j\in \sigma(j)$ for every $j\in J$.

Such a set system presentation $\present$ can be turned into a presentation in our sense by replacing each element $[n]\setminus F \in \present$ by $\overline e_F$ where 
\begin{equation}\label{eq:overline e}
(\overline e_J)_j = \begin{cases}\infty & j\in J \\ 0 & j\not\in J\end{cases}
\end{equation}
In the corresponding $\{0,\infty\}$-matrix $A$, we have that $\pi(A)_B = 0$ if there is matching from $B$ and $\infty$ otherwise. 
Conversely, given a transversal valuated matroid $V=\pi(A)$, the multiset consisting of the set of finite entries of each row of $A$ is a presentation of $\underm V$.

%

We caution readers of the literature on transversal matroids
that most authors allow the set system presenting a rank~$d$ matroid to contain more than $d$ sets.
These authors would say that all our presentations are ``of rank~$d$''.

Here is a necessary condition for transversality of valuated matroids.
\begin{proposition}[{Fink, Rinc\'on \cite[Corollary 5.6]{FR}}]\label{prop:FR local}
Let $V$ be a transversal valuated matroid. Then every matroid $M \in \MM(V)$ such that $P_M$ is a facet of $P_V$ is transversal.
\end{proposition}
In \Cref{coro_transv} we show that this condition is also sufficient.

\subsection{The set of presentations of a matroid}


Given a set system presentation $\present$ of $M$, we have that $[n]\setminus A$ is a flat of $M$ for every $A\in \present$ 
(this follows, for example, from Lemma~\ref{lem:containment}).
So, to characterize the presentations of $M$ is to determine when a multiset of $d$ flats of $M$ constitutes the complements of a presentation of $M$.
This problem was solved by Brualdi and Dinolt \cite{BrualdiDinolt} who proved that every transversal matroid $M$ has a unique maximal presentation 
and showed how to derive all other presentations from it. 
To describe the unique maximal presentation they use an algorithm which we now discuss.

Let $\mu$ be the M\"obius function on the lattice of cyclic flats $\CF(M)$. For $F\in \CF$ define
\begin{equation}\label{eq:definition of tau}
\tau(F) := \sum\limits_{F'\in \CF(M), \enspace F \subseteq F'} \mu(F, F') \cork(F').
\end{equation}

If $\tau$ is non-negative, we can consider the multiset of cyclic flats $\DF(M)$ where each $F\in \CF(M)$ has multiplicity $\tau(F)$. 
Brualdi calls this the \emph{distinguished family of cyclic flats} \cite[p.~77]{BrualdiWhite}.
\begin{proposition}[Brualdi and Dinolt \cite{BrualdiDinolt},, Theorem 4.7]
\label{BrualdiDinolt}
Let $M$ be a transversal matroid. Then $\tau$ is non-negative, 
and the complements of the distinguished family of cyclic flats make up the unique maximal presentation of $M$. 
Moreover, $\AAA= \ldb A_1, \dots, A_d\rdb$ is a presentation if and only if the complements are flats $F_i= [n]\setminus A_i$ such that 
\[
\ldb \cocl(F_1),\dots, \cocl(F_d)\rdb = \DF(M)
\]
and for every $I\subseteq [d]$
\[
\cork(\bigcap_{i\in I}F_i) \ge |I|.
\] 
\end{proposition}
At the heart of this paper is the idea of generalizing the above result to valuated matroids.

The literature contains several statements similar or equivalent to the above.
Below we describe another reformulation of \Cref{BrualdiDinolt} as a precise bijection between integer vectors and presentations.
See Bonin~\cite{BoninTransversalNotes} for more detail on the equivalence.

\begin{proposition}\label{prop:Bonin}
Let $M$ be a matroid, and $\beta:\FF(M)\to\mathbb Z$.
Then $M$ has a transversal presentation consisting of $\beta(F)$ copies of $[n]\setminus F$ for each $F\in\FF(M)$
if and only if $\beta$ satisfies the following inequalities:
\begin{align}
\beta(F)\geq0 &\quad\mbox{for all $F\in\FF(M)$}\label{eq:beta pos}\\
\sum_{G\geq F}\beta(G)\leq\cork(F) &\quad\mbox{for all $F\in\FF(M)$}\label{eq:beta F}\\
\sum_{G\geq F}\beta(G)=\cork(F) &\quad\mbox{for all $F\in\CF(M)$}.\label{eq:beta CF}
\end{align}
\end{proposition}

Notice that if $M$ is a transversal matroid,
extending $\tau$ to be 0 for every non-cyclic flat yields a solution of the integer program in \Cref{prop:Bonin}.
This is the minimal such function in the following sense:
if $\beta$ is a solution of this system for some matroid~$M$, then by \Cref{BrualdiDinolt} we have that for every $F\in\CF(M)$ 
\[
\sum_{\cocl(G)=F}\beta(G) = \tau (F).
\]

Testing if $M$ is transversal can be done by checking whether $\tau$ (as defined in \Cref{eq:definition of tau}) satisfies inequalities \eqref{eq:beta pos} and~\eqref{eq:beta F}. Another test for transversality, \Cref{prop_transv_charac}, was provided by Mason and Ingleton.

The above discussion shows that every set system presentation of~$M$
can be obtained from the maximal presentation by replacing some elements $F$ with $G$ where $\cocl(G)=F$.
Therefore, every set system presentation of~$M$ is obtained from the maximal presentation by adding relative coloops to the flats chosen.

\begin{example}
\label{ex:uniform}
The work \cite{FR} focuses on presentations of valuated matroids $V$ with no $\pl VB=\infty$,
which it represents as matrices like $A$ in \Cref{def:presentation}.

The underlying matroid of any such~$V$ is the \newword{uniform matroid} $U_{d,n}$,
the matroid with $\BB(U_{d,n})=\binom{[n]}d$.
The only cyclic flats of $U_{d,n}$ are $\emptyset$ and $[n]$,
so we get $\tau([n])=0$ (as is the case for all matroids) and $\tau(\emptyset) = d$.
Hence the maximal presentation of $U_{d,n}$ is $\ldb\underbrace{[n],\ldots,[n]}_d\rdb$.

The non-cyclic flats of~$U_{d,n}$ are all sets $F$ such that $0<|F|<d$.
Inequality~\eqref{eq:beta F} says that for any $J\subseteq[n]$ with $|J|<d$, 
there cannot be more than $d-|J|$ sets among the complements of a presentation of~$U_{d,n}$
that are supersets of or equal to~$J$.
Because a proper flat of $U_{d,n}$ has at most $d-1$ elements, the case $|J|=d$ of the last sentence is true as well.
\Cref{prop:Bonin} says that any set system of $d$ sets satisfying these conditions is a presentation of~$U_{d,n}$.
After translating to matrices via equation~\eqref{eq:overline e},
this is the statement (c)$\Leftrightarrow$(d) of \cite[Proposition 8]{FR}.
The reader may check that when $n=d$ one recovers Philip Hall's marriage theorem,
and when $n=d+1$, the dragon marriage theorem of Postnikov~\cite{PostPAB}.
\end{example}

\begin{example}
\label{ex:no transversal}
Consider the matroid $M$ on 6 elements of rank~2 given by $\BB(M) := {6\choose 2}\setminus\{12,34,56\}$. 
For $M$ to have a transversal presentation, $\beta$ would have to satisfy $\beta(12)=\beta(34)= \beta(56)= 1$, as all of the sets $12$, $34$, $56$ are cyclic flats of corank~$1$. 
But this means that $\sum\limits_{F\ge \emptyset}\beta(F) \ge 3 > \cork(\emptyset) = 2$, which is a violation of condition \eqref{eq:beta CF}. 
In consequence, no valuated matroid $V$ such that $M\in \MM(V)$ can be in the image of the Stiefel map.

\begin{figure}[htb]
	\centering
		\includegraphics[scale=0.5]{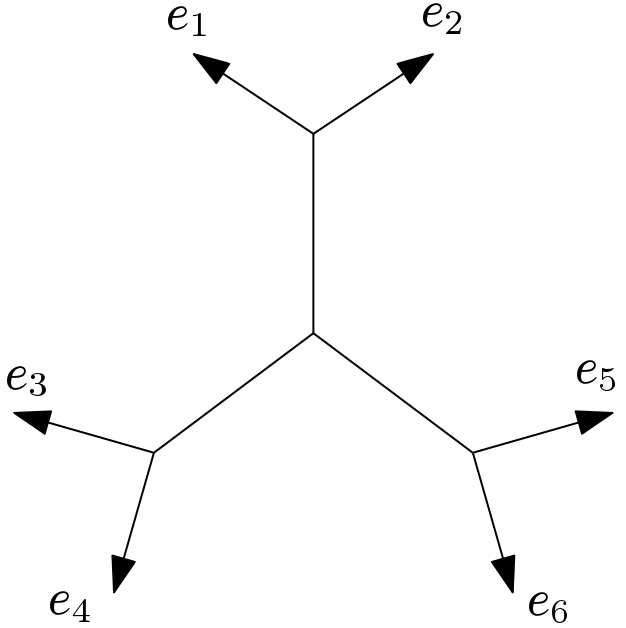}
	\caption{The `snowflake' tropical linear space, where $V_{12} = V_{34} = V_{56}=1$ and $V_B = 0$ for $B \in {6\choose 2}\setminus \{12,34,56\}$, does not correspond to a transversal valuated matroid.}
	\label{fig:snowflake}
\end{figure}

Similar reasoning shows that no rank~2 matroid with three or more nontrivial parallel classes has a transversal presentation.
The non-transversality of a valuated matroid can be seen in the geometry of the corresponding linear space. 
For example, the tropical linear space in \Cref{fig:snowflake} has a vertex incident to 3 bounded edges.
This vertex corresponds to the non-transversal matroid $M$ and each bounded edge corresponds to one of its non trivial cyclic flats.
This provides one proof that the tree formed by the bounded faces of a Stiefel tropical linear space of rank~$2$ is a path. 

\end{example}

\subsection{Additional remarks}

\begin{remark}\label{rem:tGrassmannian}
The image of~$\stiefel$ is always contained in the tropical Grassmannian $\operatorname{TropGr}(d,n)$, 
the tropicalization of the Grassmannian over a field in its Pl\"ucker embedding~\cite{SS}. 
The matroid of~\Cref{ex:no transversal} lies in the tropical Grassmannian for any field,
so $\stiefel$ does not surject onto $\operatorname{TropGr}(d,n)$.
\end{remark}

\begin{remark}
A family of presentations that have been the focus of much previous work are the \emph{pointed} presentations,
where $A$ has a tropical identity matrix as a maximal submatrix
\cite{HJS,Rincon,JoswigLoho}.
The unvaluated matroids with pointed presentations are called
\emph{fundamental transversal matroids} \cite[Section 3.1]{BoninTransversalNotes}
(see also \cite{Bix77,RI80});
by \Cref{prop_zoom_pres}, these presentations can be taken to be by $\{0,\infty\}$ matrices.
If $V$ has a pointed presentation $A$, then all facets of~$P_V$ share the vertex $e_J$
where $A_J$ is the identity submatrix.
The converse is false: 
for example, non-fundamental transversal matroids exist, and for these $P_V$ has only one facet.
In other words, 
whereas the Grassmannian $\operatorname{Gr}(d,\mathbb K^n)$ over a field $\mathbb K$ has an atlas of charts isomorphic to $\mathbb A_{\mathbb K}^{d(n-d)}$,
one for each position of the identity submatrix,
the corresponding maps from $\mathbb T^{d(n-d)}$ fail even to cover the image of~$\stiefel$.
\end{remark}

\begin{remark}\label{rem:stable sum}
If $V$ and $V'$ are valuated matroids on~$[n]$ of respective ranks $d$ and~$d'$,
their \emph{stable sum} $V+V'$ is the valuated matroid of rank $d+d'$ defined by
\[\pl{(V+V')}J = \min\{\pl VB+\pl{V'}{B'} : \textstyle B\in\binom{[n]}d,B'\in\binom{[n]}{d'},B\cup B'=J\}\]
for each $J\in\binom{[n]}{d+d'}$,
provided that $\pl{(V+V')}J<\infty$ for some~$J$.
Stable sum generalizes matroid union in the special case that the matroid union is additive in rank,
for which reason Frenk~\cite[Section 4.1]{Frenk} calls it the ``valuated matroid union''.
In this language, presentations are decompositions of a valuated matroid as a stable sum of rank~1 valuated matroids.
\end{remark}

\begin{remark}
A way of looking at the tropical Stiefel map which we do not take up here is in terms of the semimodule theory of~$\TT$.
This viewpoint is adopted in \cite{CGM},
and is generalized in \cite{Mundinger} to the valuated version of Perfect's ``induction'' of a matroid across a directed graph \cite{Perfect}.
\end{remark}

\section{Characterizing presentations by regions}\label{sec:regions}

In this section, we characterize presentations of a valuated matroid $V$
in terms of bounds on the number of points which may lie in certain regions of~$\Be V$.

We start by noting that the search for transversal presentations of a tropical linear space $L$ is helpfully delimited by the fact that all elements of a presentation must lie in~$L$.
This is essentially the tropical Cramer rule \cite{TCDR,FirstSteps},
but the proof is short so we include it for convenience.

\begin{lemma}\label{lem:containment}
Let $\ldb A_1,\ldots,A_d\rdb$ be a transversal presentation of a valuated matroid $V$.
Then $A_i\in \Be V$ for each $i\in[d]$.
\end{lemma}

\begin{proof}
Write the presentation as a matrix $A\in\TT^{d\times n}$.
Define an expanded matrix $A^{(i)}$
whose first $d$ rows agree with~$A$ and whose $(d+1)$\/st row equals its $i$\/th row.
Given a set $C\in\binom{[n]}{d+1}$, 
let $(j(i') : i'\in[d+1])$ be a transversal from $[d+1]$ to~$C$ in~$A^{(i)}$
so that $\sum_{i'} A^{(i)}_{i',j(i')}$ is minimal.
By construction of $A^{(i)}$,
swapping the $i$\/th and $(d+1)$\/th entries of the transversal preserves this sum.
This implies that both $k=j(i)$ and $k=j(d+1)$ minimize the quantity $A_{i,k}+\pl L{C\setminus\{k\}}$,
because in each case $\pl L{C\setminus\{k\}}$ is the sum of the matrix entries in the transversal other than
the entry in the $(d+1)$\/th row, which contributes $A_{i,k}$.
Therefore the tropical equations in the definition of~$\Be V$ hold at $A_i$.
\end{proof}

Our next step is to generalize Proposition~\ref{prop:Bonin},
which characterizes set system presentations of matroids,
to describe presentations of unvaluated matroids by points with unrestricted tropical coordinates.
In this case, the regions we invoke can be seen as generalizing
the ranges of summation in inequalities \eqref{eq:beta F} and~\eqref{eq:beta CF}.

For that purpose we define \emph{relative support}.
This is essentially the same notion as covectors in the theory of tropical hyperplane arrangements \cite[Section~3]{ArdilaDevelin}.
The covector of a point is the list of complements of
its relative supports with respect to the apex of each tropical hyperplane.

\begin{definition}
\label{def:relsupp}
Let $x$ and $y$ be two points in~$\TP$ such that $x$ has finite coordinates. The \emph{relative support} $\relsupp_x(y)\subseteq [n]$ of $y$ with respect to~$x$
is the set indexing the coordinates where $y-x$ does not attain its minimum.
\end{definition}

Note that addition of a scalar multiple of $(1,\ldots,1)$ to the coordinates of a point does not affect its relative support, so the relative support is well defined. 
If $x$ has a fixed vector of affine coordinates $(x_1,\ldots,x_n)\in\RR^n$,
then we say that the \emph{supportive} choice of affine coordinates $(y_1,\ldots,y_n)$ for~$y$,
with respect to $(x_1,\ldots,x_n)$,
is the one which achieves $\min_j(y_j-x_j)=0$.  
In terms of supportive coordinates, \Cref{def:relsupp} becomes
\[\relsupp_x(y) = \{j\in[n]:y_j>x_j\}.\]

Let $L = \Be M$ where $M$ is a matroid of rank $d$ on~$[n]$.
By definition of $L$, we have that $\relsupp_0(y)\in \FF(M)$ for every $y\in L$. 
So for each flat $F\in \FF(M)$ we define the region
\[
R_0(F,L) := \{y \in L : F\subseteq \relsupp_0(y)\}.
\]
In supportive coordinates with respect to the zero vector, 
$R_0(F,L)$ consists of all the points which have positive entries in the coordinates indexed by $F$. 
Similarly, for each cyclic flat $F\in \CF(M)$ we define another region
\[
R_\infty(F,L) := \{y \in L : \forall j\in F, \enspace y_j = \infty\}.
\]
In other words, $R_0(F,L)$ consists of all points $y$ in $L$ where no coordinate of $y$ in $F$ achieves the minimum among its coordinates and $R_\infty(F,L)$ are those points in $L$ whose coordinates in $F$ are $\infty$. Clearly $R_0(F,L)\subseteq R_\infty(F,L)$. Given a multiset of $d$ points in $L$, $\present=\ldb A_1,\dots, A_d\rdb$, we define the numbers
\begin{align*}
\sigma_0(\present,F) &:= |\{i\in [d] : A_i \in R_0(F,L)\}| \\
\sigma_\infty(\present,F) &:= |\{i\in [d] : A_i \in R_\infty(F,L)\}|
\end{align*}
where $F$ is a flat in the first line, and a cyclic flat in the second.

\begin{proposition}
\label{prop_A}
Let $M$ be a transversal matroid, $L = \Be M$ and $A_1,\dots, A_d \in L$. Then $\present = \ldb A_1,\dots, A_d\rdb$ is a presentation of $M$ if and only if the following conditions hold:
\begin{enumerate}
	\item $\forall F\in \FF(M), \enspace \sigma_0(\present,F) \le \cork(F)$. \label{cond1}
	\item $\forall F\in \CF(M), \enspace \sigma_\infty(\present,F) = \cork(F)$. \label{cond2}
\end{enumerate}
\end{proposition}
\begin{proof}
Let $A\in \TT^{d\times n}$ be the matrix whose rows are the supportive coordinates for $A_1,\dots, A_d$ with respect to~0,
so all entries are nonnegative and each row contains a zero.
First we assume that $\ldb A_1,\dots, A_d\rdb$ is a presentation of $M$, that is $\stiefel(A) = M$. 
Let $F\in\FF$ and suppose that condition (\ref{cond1}) is not satisfied for $F$. Let $k=\cork(F)$. Let $B\in \BB(M)$ such that $|F\cap B| = d-k$. There are $k+1$ rows with positive coordinates in all of the columns indexed by $F$. This means that in the square $d\times d$ submatrix given by the columns of $B$, there is a $(k+1)\times (d-k)$ submatrix whose entries are all positive. Then the tropical minor corresponding to $B$ must be positive, which is a contradiction as $\pl MB = 0$.

Now suppose there is a cyclic flat $F\in\CF(M)$ that violates condition (\ref{cond2}). As we already proved condition (\ref{cond1}) is satisfied, we can assume $\sigma_\infty(\present,F) < \cork(F) = k$. 
Then there are $d-k+1$ rows with finite entries in the columns corresponding to $F$. 
Assume there is a matching of the submatrix of $F$ with these rows.
Then any matching of the whole matrix can be used to get a matching that uses the columns of $F$ in all of those $d-k+1$ rows by exchanging the entries. 
This is a contradiction to the rank of~$F$; 
so no such matching exists, and there must be a violation of Hall's condition. 
Let $I$ be the violating subset of rows of size $m$, so that there are at most $m-1$ columns with which elements of $I$ can be matched. 
Let $j$ be one of those columns. Because $F$ is cyclic there should be a matching of $d-k$ rows to $F-j$. So there is a row $i$ corresponding to a point in $R_\infty(F,L)$ which is not used in this matching. Then $I-i$ has access to at most $\leq m-2$ columns of $F-j$, which is a contradiction to the matching. 

We now do the other direction. Assume conditions (\ref{cond1}) and (\ref{cond2}) are satisfied. 
Because $A_i\in L$, we have $\relsupp_0(A_i)\in \FF(M)$. 
Consider the initial matroid $M'= \stiefel(A)^{\mathbf{0}}$, 
that is, the matroid whose bases are given by the entries where $\stiefel (A)$ is $0$. 
This $M'$ is transversal, and Condition (\ref{cond1}) implies that all independent sets in $M$ are also independent sets in~$M'$ (see Lemma 4.4 in \cite{BrualdiDinolt}).
This means that for each $B\in\BB(M)$ there is a matching on the 0 entries of $A$, so that $B\in M'$. 

Now let $B \in \binom{[n]}{d} \setminus \BB(M)$. Then there exists $F\in \CF(M)$ of rank $k$ such that $|B\cap F| > k$. 
By condition (\ref{cond2}) there are $d-k$ rows with infinity entries at the columns of $F$. 
This means that in the square submatrix of $A$ with columns indexed by $B$, 
there is a $(k+1)\times (d-k)$ submatrix with all entries infinity. So $\stiefel(A)_B = \infty$. 
Altogether, this shows $\stiefel (A) = M$.
\end{proof}

\medskip

We now turn our attention to the more general case $L = \Be V$ where $V$ is any valuated matroid. 
When we look at general tropical linear spaces, we have to define the regions $R_0$ and $R_\infty$ more carefully. 
They will now have three parameters: the tropical linear space $L= \Be V$, a point $x\in L$ with finite coordinates and a flat $F\in \FF(M)$ 
such that the relative interior of~$\f LM$ contains $x$. 
Before we define these regions, we provide the following lemma which explains why it still makes sense to take flats as parameters.

\begin{proposition}
\label{prop:flat relsupp}
Let $L= \Be V$ be a tropical linear space, $M\in \MM(V)$ and $x$ be a point in the relative interior of $\f LM$. 
Then $\relsupp_x(y) \in \FF(M)$ for any $y\in L$.
\end{proposition}
\begin{proof}
Notice that $x$ being in the relative interior of $L_M$ already implies that $x$ has finite coordinates, so it makes sense to talk about $\relsupp_x(y)$. Without loss of generality we can translate $L$ so that $x$ is the origin. 
In this case, we may assume that $\pl VB = 0$ if and only if $B \in \BB(M)$. 
Now suppose that there exists $y\in L$ such that $\relsupp_x(y) \notin \FF(M)$. This means there is an element $i\in [n] \setminus \relsupp_x(y)$ such that $i \in \cl_M(\relsupp_x(y))$. 
Let $B\in \BB(M)$ be such that $|B\cap \relsupp_x(y)| = \rk_M(\relsupp_x(y))$. 
Then $i\notin B$, and $B\cup \{i\}\setminus \{j\} \notin \BB(M)$ for any $j\in B\setminus\relsupp_x(y)$. 
By the tropical Pl\"ucker equation corresponding to $B\cup \{i\}$, the minimum in 
\[
\min\limits_{B'\cup\{j\} = B\cup\{i\}} \pl V{B'}+y_j
\]
is achieved twice. We have that $\pl VB +y_i = 0$. 
But for any other $B'\cup\{j\}$, if $j\in \relsupp_x(y)$ then $y_j> 0$ and if $j \notin \relsupp_x(y)$ then $\pl V{B'}>0$. 
So the minimum is only attained once, which is a contradiction.
\end{proof} 
Given a tropical linear space $L= \Be V$, a matroid $M\in\MM(V)$, a flat $F\in \FF(M)$ and a point $x\in \relint(\f LM)$, we define two regions, which we will use to constrain the possible position of points in presentations.
Let
\[
R_0(F,x,L) := \{y \in L : F\subseteq \relsupp_x(y)\},
\]
and, whenever $F\in \CF(M)$,
\[
R_\infty(F,x,L) := \bigcap\limits_{y\in \relint \left(L_{M|F\oplus M/F}\right)}R_0(F,y,L).
\]
See \Cref{ex:thmA} for examples of these definitions.


\begin{lemma}
\label{lem:region equiv}
Let $M$ be a matroid. Then 
\begin{enumerate}
	\item $R_0(F,0,\Be M)= R_0(F,\Be M)$
	\item $R_\infty(F,0,\Be M) = R_\infty(F,\Be M)$
\end{enumerate}
where $R_0(F,\Be M)$ and $R_\infty(F,\Be M)$ are the regions defined earlier.
\end{lemma}
\begin{proof}
The first equivalence is straight forward from the definitions of $R_0(F,0,\Be M)$ and $R_0(F,\Be M)$.
To see that $R_\infty(F,0,\Be M) = R_\infty(F,\Be M)$ note that $\relint \left(\f{\Be M}{M|F\oplus M/F}\right) \subseteq R_0(F,\Be M)$, so every $y \in \relint \left(\f{\Be M}{M|F\oplus M/F}\right)$ has positive entries in $F$ when written in supportive coordinates with respect to $0$. 
Any $z\in R_0(F,y,\Be M)$ must have coordinates larger than $y$ in $F$ when written in supportive coordinates with respect to the $0$.
As $\relint \left(\f{\Be M}{M|F\oplus M/F}\right)$ is an open cone, $y$ can have arbitrarily large coordinates in $F$ and any $z\in R_\infty(F,0,\Be M)$ must have infinite entries at $F$, 
so $R_\infty(F,0,\Be M) \subseteq R_\infty(F,\Be M)$. 
But clearly also $R_\infty(F,\Be M)\supseteq R_0(F,y,\Be M)$ for every $y\in \relint \left(\f{\Be M}{M|F\oplus M/F}\right)$, so the equality holds.
\end{proof}

Given a multiset $\present = \ldb A_1,\dots,A_d\rdb$ of $d$ points in $L$ we can define $\sigma$ as in the unsubdivided case. For $x \in \relint(\f LM)$,
\begin{align*}
\sigma_0(\present,F,x) &:= |\{i\in [d] : A_i \in R_0(F,x,L)\}|\\
\sigma_\infty(\present,F,x) &:= |\{i\in [d] : A_i \in R_\infty(F,x,L)\}|
\end{align*}
where $F$ is a flat of~$M$ in the first line, and a cyclic flat of~$M$ in the second.
The following lemma shows that $R_\infty(F,x,L) \subseteq R_0(F,x,L)$ and $\sigma_0(\present,F,x) \ge \sigma_\infty(\present,F,x)$ for every vertex $x$ of $L$.

\begin{lemma}
\label{lemma:r0_containment}
Let $M\in \MM(V)$ be a connected matroid, $F\in \CF(M)$ and $y\in \relint(L_{M|F\oplus M/F})$. Then $R_0(F,y,L) \subseteq R_0(F,\vertex LM,L)$.
\end{lemma}
\begin{proof}
If $y\in\relint(L_{M|F\oplus M/F})$, 
then $y$ is of the form $\vertex LM+c_1e_{F_1}+\dots+c_ke_{F_k}$ for a flag $F_1\subset\cdots\subset F_k$ containing $F$ and such that $0\le c_i<\infty$ for every $i$;
the $c_i$ are finite because we have excluded faces at infinity from the relative interior.
This is the same form as points have in the cone $\f {\Be M}F$ of the Bergman fan of~$M$. 
This means in particular that for any $j\notin F$ and $j'\in F$
we have $y_j\le y_{j'}$ when written in the supportive coordinates with respect to (fixed coordinates for) $\vertex LM$. 
So if $z\in R_0(F,y,L)$, then there is a $j\notin F$ such that $j\notin \relsupp_y(z)$. 
For every $j'\in F$ it follows that
$(z-y)_{j'}>(z-y)_j$, and $(y-\vertex LM)_{j'} \ge (y-\vertex LM)_j$, so $(z-\vertex LM)_{j'}>(z-\vertex LM)_j$ which means that $z\in R_0(F,\vertex LM,L)$.
\end{proof}

The following definition helps us use the Bergman fan case for the more general setting of tropical linear spaces.
\begin{definition}
Let $L = \Be V$ be a tropical linear space, $M\in \MM(V)$ and $x \in \relint(\f LM)$. The \newword{zoom} map of $L$ to $x$ is the map $Z_x : L \rightarrow \Be M$ such that
\[
Z_x(y)_j := \begin{cases}
0 & \text{ when } j \notin \relsupp_x(y)\\
\infty & \text{ when } j \in \relsupp_x(y)
\end{cases}
\]
\end{definition}

We think of $Z_x$ as `zooming' into $x$, pushing all points of $L$ away from $x$ to infinity in a straight line. Thus, $Z_x(L)$ keeps only local information of $L$ around $x$. 

\begin{proposition}
\label{prop_zoom_pres}
Let $M\in \MM(V)$ be a coloop-free matroid, not necessarily connected, and let $x$ be a point in the relative interior of $\f LM$.
Suppose $\present = \ldb A_1,\dots, A_d\rdb$ is a presentation of $V$. 
Then $Z_x(\present) = \ldb Z_x(A_1),\dots, Z_x(A_d)\rdb$ is a presentation of~$M$,
i.e. $\ldb [n]\setminus \relsupp_x(A_1) \dots \break [n]\setminus \relsupp_x(A_d)\rdb$ is a set system presentation of $M$.
\end{proposition}
The corresponding arguments in \cite{FR} are Propositions 5.5 and~5.9.

\begin{proof} 
Let $A\in \TT^{d\times n}$ be the matrix whose $i$\/th row consists of $A_i$ written in supportive coordinates with respect to $x$. 
Notice that the scaling of rows in the matrix $A$ does not change $\stiefel(A)$ and adding the vector $x$ to each of the rows of $A$, changes $\stiefel(A)_B$ by adding $\sum\limits_{j\in B} x_j$. 
This implies that $y\in \Be{\stiefel(A)}$ if and only if $y+x\in L$. 
So we have that $\Be {\stiefel(A)}$ equals $L -x$, the tropical linear space $L$ translated so that $x$ is at the origin.

Tropically exponentiating (i.e.\ classically multiplying) each entry of $A$ by~$t$
transforms $L-x$ by a classical homothety centered at the origin of factor $t$, 
so $\Be {\stiefel(A^t)} = t(L-x)$. 
When $t \rightarrow \infty$, we have that $A^t\rightarrow Z_x(A)$ where $Z_x(A)$ is the matrix where the row $i$ is given by $Z_x(A_i)$. 
Since tropical linear spaces are locally fans, we have that as $t \rightarrow \infty$, $t(L-x)$ goes to the fan with which $L-x$ coincides near the origin.
This is the same fan whose translation by $x$ coincides with $L$ near $x$, namely $\Be M$, since $x \in \relint(\f LM)$.
Because $\stiefel$ is a continuous map in its domain, 
these two limits imply that $\stiefel(Z_x(A)) = \Be M$ as long as $Z_x(A)$ is still in the domain of~$\stiefel$. 
So the only thing left to prove is that this is the case, namely, that there is a set $B$ for which $\pi_B(A) = 0$. 

If there were no maximal minor of $A$ equal to~$0$, 
then there would be an $a\times b$ submatrix $A'$ of $A$ consisting of strictly positive entries such that $a+b>n$. 
Among such matrices $A'$ select one where $b$ is maximal, i.e.\ with the most columns. 
Let $I$ be the set of rows taken by $A'$ and $J$ be the set of columns not taken by $A'$. Notice that $|I| = a > n-b = |J|$. 
Consider a bipartite graph $G$ whose vertices are $I\amalg J$ and containing the edge $(i,j)$ just if $A_{i,j}= 0$. 
If $G$ is disconnected, then there is a connected component with vertices $I'\subseteq I$ and $J'\subseteq J$ with $|I'|>|J'|$. So the submatrix of $A$ given by rows $I$ and columns $[n]\setminus J'$ is strictly positive and has more columns than $A'$, which is a contradiction. So $G$ is connected. 

Let $j\in J$. As $M$ has no coloops, then there is a basis $B\in \BB(M)$ such that $j\notin B$. 
Because $0\in \Be {\stiefel(A)}_M$, then $\stiefel(A)_B$ is minimal among all maximal minors of $A$. The value of $\stiefel(A)_B$ is achieved by a matching $\sigma: B \rightarrow [d]$. 
All matching must use an entry of $A'$, because $a+b>n$ implies that the total number of columns and rows of $A$ not included in $A'$ is less than $d$.
So there is an element $j'\in [n]\setminus J$ such that $\sigma(j') \in I$. 
Let $G'$ be the graph where you add to $G$ the vertex $j'$ and the edge $(\sigma(j'),j')$. 
As $G'$ is connected, then there is a path $G$ from $j'$ to $j$. The matching given by $\sigma$ does not use consecutive edges. 
By replacing each edge used by $\sigma$ in $G$ by the edge that follows it, we get a matching $\sigma'$ from $B-i\cup j$ to $[d]$. 
But the weight of this matching is less than that of $\sigma$ as we replaced a strictly positive entry $A_{\sigma(j'),j'}$ by zero. 
This contradicts the minimality of $\stiefel(A)_B$.
\end{proof}

\begin{example}\label{ex:zoom}
Let $V$ be the valuated matroid of rank $3$ on 5 elements such that $V_{123}= 1$, $V_{145}= \infty$, and $V_B=0$ for any $B\in\binom{[5]}3$ other than these two.
Notice that the rows of the matrix 
\[
A = \begin{pmatrix}
0 &	0 &	0 &	0 &	0 \\
1 &	1 &	1 &	0 &	0 \\
\infty &	0 &	0 &	\infty &	\infty 
\end{pmatrix}
\]
form a presentation of $V$, that is $\pi(A) = V$. Let $x = A_2\in \TT\PP^4$ be the second row of $A$. 
The matroid $V^x$ is such that $\BB(V^x) = \{B\in \binom{5}{3} : 45\not\subset B\}$.
(See also \Cref{fig_dmat}, where the same matroid $V^x$ appears as~$M_2$.)
We have that
\[\relsupp_x(A_1) = 45, \quad  \relsupp_x(A_2) = \emptyset,\quad  \relsupp_x(A_3) = 145.\] 
It is straightforward to check that the collection of flats $\ldb 45,\emptyset, 145\rdb$ satisfy the conditions of \Cref{BrualdiDinolt}, so their complements are a set system presentation of $V^x$.
In other words, the rows of the matrix
\[
Z_x(A) = \begin{pmatrix}
0 &	0 &	0 &	\infty &	\infty \\
0 &	0 &	0 &	0 &	0 \\
\infty &	0 &	0 &	\infty &	\infty 
\end{pmatrix}
\]
form a presentation of $V^x$.
\end{example}

We will need the following lemma.
\begin{lemma}
\label{lemma_zoom}
Let $M\in \MM(V)$ be a coloop-free matroid and let $x\in L_M$ lie in a coloop-free face $M$. 
For $F\in \FF(M)$ we have that
\[Z_x^{-1}(R_\infty(F,0,\Be M)) = R_0(F,x,L).\]
\end{lemma}
\begin{proof}
A point $y$ satisfies $Z_x(y) \in R_\infty(F,0,\Be M) = R_\infty(F,\Be M)$ if and only if $Z_x(y)_i = \infty$ for every $i\in F$. By definition of the zoom map $Z_x$, this happens if and only if $i\in \relsupp_x(y)$ for every $i\in F$, which is equivalent to $y \in  R_0(F,x,L)$. 
\end{proof}

\begin{proposition}
\label{prop:half B}
Let $\present$ be a presentation of $V$.
Then for any coloop-free matroid $M\in \MM(V)$ and $x\in \relint(\f LM)$
we have that $\sigma_0(\present,F,x) \leq \cork_M(F)$ for $F\in \FF(M)$, with equality if $F\in \CF(M)$.
\end{proposition}
\begin{proof}
By \Cref{prop_zoom_pres} we have that $Z_x(\present)$ is a presentation of $\Be M$. Then by \Cref{prop_A} there are at most $\cork_M(F)$ elements of $Z_x(\present)$ in $R_0(F,0,\Be M)$. By \Cref{lemma_zoom},
\[Z_x(R_0(F,x,L)) \subseteq R_\infty(F,0,\Be M) \subseteq R_0(F,0,\Be M)\]
so there are at most $\cork_M(F)$ elements of $\present$ in $R_0(F,x,L)$. If $F\in\CF(M)$ then there are exactly $\cork_M(F)$ elements of  $Z_x(\present)$ in $R_\infty(F,0,\Be M)$ so there are exactly $\cork_M(F)$ elements of $\present$ in $R_0(F,x,L)$.
\end{proof}

\begin{theorem}
\label{theorem_A}
Let $L = \Be V$ be a tropical linear space and $A_1,\ldots, A_d \in L$. Then $\present = \ldb A_1,\dots, A_d\rdb$ is a presentation of $L$ if and only if for every connected matroid $M\in\MM(V)$ the following hold:
\begin{enumerate}\renewcommand{\labelenumi}{\textup{(\theenumi)}}
	\item $\sigma_0(\present,F,\vertex LM) \leq \cork_M(F)$ for all $F\in \FF(M)$; and \label{c1}
	\item $\sigma_\infty(\present, F,\vertex LM) = \cork_M(F)$ for all $F\in \CF(M)$. \label{c2}
\end{enumerate}
\end{theorem}
\begin{proof}
Let $A$ be a presentation of a tropical linear space $L$. Applying \Cref{prop:half B} for every vertex $\vertex LM$ of $L$ gives us condition \eqref{c1}. 
For any connected matroid $M$ and every $F\in \CF(M)$, by \Cref{lemma_zoom} we have that there are exactly $\cork(F)$ elements of $A$ in $R_0(F,\vertex LM,L) = Z_{\vertex LM}^{-1}(R_\infty(F,0,\Be M))$. 
If condition (\ref{c2}) is not satisfied, it means that one of those points is in $R_0(F,\vertex LM,L) \setminus R_\infty(F,\vertex LM,L)$. Let $A_i$ be that point. 

Then there exists $y\in \f L{M|F\oplus M/F}$ such that $A_i \notin R_0(F,y,L)$. 
From $F\in\CF(M)$ we see that $M|F\oplus M/F$ is coloop-free and $F\in \CF(M|F\oplus F)$, 
so by \Cref{prop:half B} we have that $\cork_{M/F\oplus M|F}(F)=\sigma_0(\present,F,y)$. Notice also that $\cork_M(F) = \cork_{M|F\oplus M/F}(F)$. However by \Cref{lemma:r0_containment} we have that $R_0(F,y,L)\subseteq R_0(F,\vertex LM,L)$ so
\begin{align*}
\sigma_0(\present,F,y) &\le \sigma_0(\present, F,\vertex LM)-1\\
 &= \cork_M(F)-1\\
 &= \cork_{M/F\oplus M|F}(F)-1\\	
 &= \sigma_0(\present,F,y)-1
\end{align*}
which is a contradiction.

Conversely, suppose $\present$ satisfies conditions (\ref{c1}) and (\ref{c2}). 
Let $A$ be the matrix which has $\present$ as its rows, so what we have to prove is that $\stiefel(A)=V$. 
For any connected matroid $M$, we have that $Z_{\vertex LM}(\present)$ satisfies (\ref{c1}) and (\ref{c2}) for $\Be M$, so it is a presentation of $\Be M$.
By adding $\vertex LM$ to each element of $Z_{\vertex LM}(\present)$ we get a presentation of $\Be M + \vertex LM$.
The matrix we obtain by concatenating all of these presentations
coincides in its finite entries with $A$. 
As the finite Pl\"ucker coordinates of $\Be M+\vertex LM$ agree with $V$ up to adding a scalar, 
the difference between any pair of Pl\"ucker coordinates of $\stiefel(A)$ both  indexed by elements of~$\BB(M)$
has the value called for by~$V$.
Because the incidence graph of edges and maximal cells in~$P_V$ is connected,
we conclude that all finite Pl\"ucker coordinates of $\stiefel(A)$ agree with~$V$ up to a single global scalar.

Let $B$ be a nonbasis of~$\underm V$. Consider a facet $Q$ of $P_{\underm V}$ such that $e_B$ fails to satisfy its defining inequality. 
Let $P_M$ be one of the maximal cells of~$P_V$ which have a facet contained in~$Q$,
and let $F$ be the cyclic flat that defines that facet. 
Then $|B\cap F| > \rk_M(F)$. As the polytope of $P_{M/F\oplus M|F}$ is in the boundary of~$P_{\underm V}$, 
we have $\sup\{z_j : z\in \f L{M/F\oplus M|F}\} = \infty$ for all $j\in F$.
This implies that points in $R_\infty(F,\vertex LM,L)$ have $\infty$ entries in the coordinates corresponding to~$F$.
Because of (\ref{c2}) for $M$ and $F$, there are $\cork(F)$ elements of $\present$ in $R_\infty(F,\vertex LM,L)$. 
So at most $\rk_M(F)$ of the rows of~$A$ 
contain a finite entry in a column indexed by $B\cap F$. 
This is a violation of Hall's condition, so there is no matching for $B$ using finite entries of $A$. So $\stiefel(A)_B= \infty$.
\end{proof}

\begin{example}
\label{ex:thmA}
Consider the tropical linear space $L = \Be V$ from \Cref{ex1}. 
There are two connected matroids in $\MM(V)$, namely
$M_1$ whose vertex in $L$ is $v_1 = [0:0:0:0]$ with bases $\BB(M_1) = \binom{4}{2}\setminus\{34\}$ 
and $M_2$ whose vertex in $L$ is $v_2 = [0:0:1:1]$ with bases $\BB(M_2) = \binom{4}{2}\setminus\{12\}$. 
Since $R_0(\emptyset,x,L) = R_\infty(\emptyset,x,L) = L$,
the conditions imposed by \Cref{theorem_A} for $F= \emptyset$ are trivial. 
We name the 4 rays in $L$:
\begin{align*}
L_1 &:= \{[a:0:0:0] : a\ge 0\} & L_2 &:= \{[0:a:0:0] : a\ge 0\} \\
L_3 &:= \{[0:0:a:1] : a\ge 1\} & L_4 &:= \{[0:0:1:a] : a\ge 1\}.
\end{align*}
We have
\begin{align*}
R_0(1,x_1,L) &= L_1 &R_0(2,x_1,L) &= L_2 \\
R_0(34,x_1,L) &= L\setminus (L_1\cup L_2) &R_\infty(34,x_1,L) &= L_3\cup L_4 \\
R_0(3,x_2,L) &= L_3 &R_0(4,x_2,L) &= L_4 \\
R_0(12,x_2,L) &= L\setminus (L_3\cup L_4)  &R_\infty(12,x_2,L) &= L_1\cup L_2 .
\end{align*}

Condition (\ref{c2}) of \Cref{theorem_A} says that any presentation has exactly one point in $L_1\cup L_2$ (the blue region in \Cref{fig_ex1}) and exactly one point in $L_3\cup L_4$ (the red region in \Cref{fig_ex1}), just as we said in \Cref{ex1}. 
Condition (\ref{c1}) says that there is at most one point in $L_i$ for every $i\in[4]$, and at most one point in $L\setminus (L_1\cup L_2)$ and in $L\setminus (L_3\cup L_4)$, but in this case this follows from condition (\ref{c2}).
\end{example}

\medskip

We end this section by using the previous theorem to understand how presentations behave under contractions.
\begin{proposition}
\label{prop_contractions}
Let $\AAA$ be a presentation of $V$ and $F\in \CF(\underm V)$ a cyclic flat of rank $k$. 
Then there are exactly $d-k$ points in $\AAA$ all of whose coordinates indexed by elements of $F$ are $\infty$. 
The projection of these points to the $[n]\setminus F$ coordinates form a presentation of $V/F$.
\end{proposition}
\begin{proof}
As $F\in \CF(\underm V)$, there are coloop-free matroids in $\MM(V)$ such that their polytopes are contained in the hyperplane 
\[H_F :=\left\{\sum\limits_{j\in F}x_j = k\right\}.\] 
Condition (\ref{cond2}) of \Cref{theorem_A} applied to any of these matroids 
implies that there are exactly $d-k$ points of~$\AAA$ with $\infty$ in the $F$ coordinates, 
because the cells of $L$ corresponding to these cells 
extend to infinity in the $e_F$ direction. Let $\AAA_F\subseteq \AAA$ be the multiset of those points. 

For every coloop-free matroid in $M' \in \MM(V/F)$ there is a coloop-free matroid $M\in\MM(V)$ such that $M/F=M'$ and $P_M\subseteq H_F$. In particular, $F\in\CF(M)$.
For every point $x'\in \Be {V/F}_{M'}$ there is a point $x\in L_M$ which coincides with $x'$ in the $[n]/F$ coordinates and is arbitrarily large in the $F$ coordinates. For such points and for any flat $F\subseteq F'\in \FF(M)$ we have that 
\[
R_0(F',x,L)\cap\left\{y_j=\infty : j\in F\right\}= \iota_F(R_0(F',x',\Be{V/F}))
\]
where $\iota_F$ again means the inclusion $\Be{V/F} \to L$ 
which sets the $F$ coordinates to~$\infty$. 
As the lattice of flats of $M'$ is isomorphic to the interval above~$F$ in lattice of flats of~$M$, 
the conditions that \Cref{theorem_A} imposes on $\AAA_F$ when applied to~$V$
are exactly the same as its conditions for presentations of $V/F$.
\end{proof}

%
%

\section{Matroid valuations}\label{sec:matroid valuations}
We will make use of the notion of matroid valuation, not to be confused with valuated matroids. 
This unfortunate similitude in names comes from the word ``valuation'' having pre-existing use in two different areas,
respectively measure theory and algebra.

Given a polyhedron $P\subseteq\mathbb R^n$, 
let $\mathbf{1}(P):\mathbb R^n\to\mathbb Z$ be its indicator function,
defined by 
\[\mathbf{1}(P)(x)=\begin{cases}
1 & x\in P \\
0 & x\not\in P.
\end{cases}\]
\begin{definition}
Let $G$ be an abelian group, and $f$ a function of a matroid taking values in~$G$.
We say that $f$ is a (\/\emph{matroid}\/) \emph{valuation} if, whenever $M_1,\ldots,M_k$ are matroids and $c_1,\ldots,c_k$ integers such that
\begin{equation}\label{eq:indicator function relation}
\sum_{i=1}^k c_i\,\mathbf{1}(P_{M_i}) = 0,
\end{equation}
it also holds that
\[\sum_{i=1}^k c_i\,f(M_i) = 0.\]
\end{definition}
For a general reference on matroid valuations, see \cite{DF}.
We recount a few basic properties here.
First, linear combinations of matroid valuations are again matroid valuations.

\begin{example}\label{ex:subdivision linear relation}
Suppose a matroid polytope $P_M$ has a subdivision 
into a collection of other matroid polytopes $Q_1,\ldots,Q_k$:
e.g.\ the regular subdivision of a valuated matroid defined in Section~\ref{ssec:valuated} is of this form.
Then by inclusion-exclusion,
\[\mathbf{1}(P_M) + \sum_{K\subseteq[k],K\neq\emptyset}(-1)^{|K|}\,\mathbf{1}\left(\bigcap_{k\in K}Q_k\right) = 0.\]
Each nonempty intersection $\bigcap_{k\in K}Q_k$ is a matroid polytope,
so discarding the terms with empty intersection gives a relation of form \eqref{eq:indicator function relation}.
Therefore such a subdivision of $P_M$ provides an ``inclusion-exclusion'' linear relation
that a matroid valuation must satisfy.
\end{example}

\begin{lemma}\label{lemma:cXr}
Let $X:X_0\subseteq\cdots\subseteq X_k$ be a chain of subsets of~$[n]$, and $r:r_0\leq\cdots\leq r_k$ nonnegative integers.
Let $c_{X,r}$ be the $\{0,1\}$-valued matroid function
which takes value 1 on~$M$ if each $X_i$ is a cyclic flat of~$M$ with $\rk_M(X_i) = r_i$ and 0 otherwise.
Then $c_{X,r}$ is a matroid valuation.
\end{lemma}

\begin{proof}
The matroid function $s_{X,r}$ which takes value 1 on~$M$ if $\rk_M(X_i) = r_i$ for each $i$, and 0 otherwise,
is known to be a matroid valuation \cite[Proposition 5.3]{DF}.
So to prove the lemma it will suffice to write $c_{X,r}$ as a linear combination of functions $s_{X',r'}$.

A set $J$ is a cyclic flat of~$M$ if and only if there is no $j\in[n]\setminus J$ such that 
$\rk(J\cup\{j\})=\rk(J)$ and no $j\in J$ such that $\rk(J\setminus\{j\})=\rk(J)-1$.
If $K\supseteq J$, then the assertion $\rk(K)=\rk(J)$ 
is equivalent to $\rk(J\cup\{k\})=\rk(J)$ for each $k\in K\setminus J$.
Therefore the indicator function of the predication
``$J$ is a flat of rank~$r$'', i.e.\
``$\rk(J)=r$ and there is no $j\in[n]\setminus J$ such that $\rk(J\cup\{j\})=r$'',
can be written by inclusion-exclusion as
\[\sum_{K\supseteq J}(-1)^{|K\setminus J|} s_{(J,K),(r,r)}.\]
Repeating the same argument in the dual allows $c_{(J),(r)}$ 
(where the two indices are lists of length one)
to be written as an alternating sum of terms $s_{(I,J,K),(r-|J|+|I|,r,r)}$.
We thus have
\begin{align}
c_{X,r}(M) &= \prod_{i=0}^k c_{(X_i),(r_i)}(M) \notag
\\&= \sum\prod_{i=0}^k(-1)^{|K_i\setminus I_i|} \, s_{(I_i,X_i,K_i),(r_i-|X_i|+|I_i|,r_i,r_i)}(M) \label{eq:c-s}
\end{align}
where the sum is over choices of sets $I_i\subseteq X_i$ and $K_i\supseteq X_i$ for each~$i$.

Submodularity implies that if $\rk(K)=\rk(J)$ for some $K\subseteq J$,
then also $\rk(K\cup L)=\rk(J\cup L)$ for every $L$ disjoint from~$K$.
Therefore, for any term of \eqref{eq:c-s} in which $K_i\not\subseteq X_{i+1}$ for some $i<k$, 
with $j\in X_{i+1}\setminus K_i$,
inserting $j$ into or removing it from~$K_k$ gives another term which is equal with opposite sign.
So we may cancel these terms, and by repeating the argument in the dual
we may impose on the index set of the sum \eqref{eq:c-s} the further conditions
$K_i\subseteq X_{i+1}$ and $I_i\supseteq X_{i-1}$.
We have furthermore that any term with $K_i\not\subseteq I_{i+1}$ is zero,
because if $j\in K_i\setminus I_{i+1}$, submodularity is violated at $X_i\cup\{j\}$ and $X_{i+1}\setminus\{j\}$.
Thus we can impose the condition $K_i\subseteq I_{i+1}$ on \eqref{eq:c-s} as well.
Under this condition all the sets in the indices form a single chain and we have
\begin{align*}
&\mathrel{\phantom=}\prod_{i=0}^ks_{(I_i,X_i,K_i),(r_i-|X_i|+|I_i|,r_i,r_i)}(M)
\\&= s_{(I_0,X_0,K_0,I_1,\ldots,K_k),(r_0-|X_0|+|I_0|,\ldots,r_k)}(M)
\end{align*}
which is a valuation.  It follows that $c_{X,r}(M)$ is a valuation.
\end{proof}

Recall the function $\tau$ defined in \Cref{eq:definition of tau}.
\begin{lemma}
\label{lemma_val}
The function $M\mapsto\tau_M(\emptyset)$ is a matroid valuation.
\end{lemma}
\begin{proof}
By Philip Hall's theorem, the M\"obius function $\mu(\emptyset,F')$ is a sum over the chains of cyclic flats from $\emptyset$ to $F'$ in $\CF$,
with a chain of length $i$ weighted $(-1)^i$.
Therefore $\mu(\emptyset,F')\cork(F')$ can be written as a linear combination of the $c_{X,r}$ running over all chains of sets $X=(X_0 = \emptyset,\ldots, X_k=F')$ and all tuples $r=(r_0,\ldots,r_k)$, the coefficient of $c_{X,r}$ being $(d-r_k)(-1)^k$.
By \Cref{lemma:cXr}, we conclude that $M\mapsto\tau_M(\emptyset)$ is a valuation.
\end{proof}

\section{The presentation space}\label{sec:presentations}

The goal of this section is to describe the set of all presentations of a given valuated matroid~$V$ (\Cref{thm_B}).
The techniques of the proof will give us further results such as \Cref{coro_transv}, the converse of \Cref{prop:FR local}:
if all facets of a regular subdivision correspond to transversal matroids,
then the subdivision defines a transversal valuated matroid.

\subsection{Distinguished matroids and apices}
We say that $V$ \emph{has transversal facets}
if all of its facets $P_V$ correspond to polytopes of transversal matroids.
So \Cref{prop:FR local} says that transversal valuated matroids have transversal facets.
Define
\[
\overline{\MM}(V) := \bigcup\limits_{F\in \CF(\underm V)}\MM(V/F).
\]
All of the matroids in this set index cells of~$P_V$.

\begin{definition}
Let $V$ be a valuated matroid with transversal facets.
The \newword{distinguished multiset of matroids} $\dmat(V)$ of $V$ contains
each matroid $M\in \overline{\MM}(V)$ with multiplicity $\tau_M(\emptyset)$. 
For any connected matroid $M\in\MM(V/F)$ with $F\in\CF(\underm V)$,
let $\vertex LM \in L=\Be V$ be the point in~$\TP$
whose coordinate vector extends $\vertex{\Be{V/F}}M$ by setting the coordinates corresponding to~$F$ to be $\infty$.  
The \newword{distinguished multiset of apices} $\dapx(L)$ of $L$ consists of $\vertex LM$ for every $M\in\dmat(V)$, with the same multiplicities.
\end{definition}
%

If $V$ has transversal facets, 
then all elements of $\overline{\MM}(V)$ are transversal, because contraction of cyclic flats preserves transversality.
To see this, notice that if $F\in \CF(M)$ the cyclic flats of $M/F$ are exactly sets of the form $S-F$ where $S$ is a cyclic flat of $M$ containing $F$. So if $\present$ is the maximal presentation of $M$, the multiset of all elements of $\present$ that are disjoint of $F$ is the maximal presentation of $M/F$ by \Cref{prop:Bonin}.
Therefore $\tau_M(\emptyset)$ only takes non-negative values for any $M\in \overline{\MM}(V)$.

\begin{proposition}
Let $V$ be a valuated matroid of rank $d$ with transversal facets. Then $|\dmat(V)| = d$.
\end{proposition}
\begin{proof}
Let us write $N(F)$ for the total number of matroids from $\MM(V/F)$ that appear in $\dmat(L)$, counted with multiplicities:
\[N(F) := \sum_{M\in\MM(V/F)}\tau_M(\emptyset).\]
If $M$ is disconnected then $\tau_M(\emptyset)=0$.  So we may freely change the coefficient of disconnected matroids in the above sum.  In particular
\begin{equation}\label{eq:N(F)}
N(F) = -\sum_{M\in\MM(V/F)}
\left(\sum_{K\in\mathcal K(M)} (-1)^{|K|}\right)
\tau_M(\emptyset)
\end{equation}
where $Q_1,\ldots,Q_k$ are the polytopes of the connected matroids in $\MM(V/F)$, and
\[\mathcal K(M) := \{K\subseteq[k] : \bigcap_{k\in K}Q_K = P_M\}.\]
The key fact being used is that if $P_M$ equals some $Q_k$ then $\mathcal K(M)=\{\{k\}\}$.
Equation \eqref{eq:N(F)} gives a case of \Cref{ex:subdivision linear relation}
which we may apply \Cref{lemma_val} to and conclude that $N(F) = \tau_{\underm V/F}(\emptyset)$.

To finish, if $F$ is a distinguished cyclic flat of~$\underm V$,
we observe that $\tau_{\underm V/F}(\emptyset) = \tau_{\underm V}(F)$,
which is its multiplicity as a distinguished cyclic flat of~$\underm V$.
So the total number of distinguished matroids of $V$, counted with multiplicity, equals the number of distinguished cyclic flats of $\underm V$, which is exactly $d$.
\end{proof}
\begin{definition}
\label{def:presentation_fan}
Let $M$ be a transversal matroid and let $t=\tau_M(\emptyset)$. The \newword{presentation fan} $\phi_M$ of~$M$
consists of all tuples of points $(p_1,\dots,p_t)\in \Be M^t$ such that $\relsupp_{0}(p_i)$ are independent flats and there is a presentation $\AAA= \ldb A_1,\dots,A_d\rdb$ of $M$ such that $A_i = [n]\setminus \relsupp_{0}(p_i)$ for $i\in[t]$. 
If $V$ is a valuated matroid with transversal facets and $L= \Be V$, then for every $M\in \dmat(V)$ we define 
\[\phi_L(M) := \phi(M)+\vertex LM\]
Finally we define the \emph{presentation space} $\Pi(L)$ of $L$ to be the orbit of 
\[
\prod\limits_{M\in \dmat(V)} \phi_L(M)
\]
under the action of $S_d$ by permuting points.
\end{definition}

In the product $\phi_L(M)$ is only taken once, regardless of the multiplicity of $M$ in $\dmat(V)$;
multiplicities are already accounted for in the definition of $\phi(M)$.
Notice that $\phi(M)$ and therefore $\phi_L(M)$ are invariant under the $S_t$ action, and $\Pi(L)$ is invariant under the $S_d$ action. 

\begin{example}
Recall the valuated matroid $V$ from \Cref{ex1,ex1.2,ex:1,ex:thmA} with connected matroids $M_1,M_2\in \MM(V)$. 
We have that $\DF(M_1) = \ldb \emptyset, \{3,4\}\rdb$ and $\DF(M_2) = \ldb \emptyset, \{1,2\}\rdb$ so $\tau(M_1) = \tau(M_2) = 1$ and $\dmat(V) = \ldb M_1, M_2\rdb$. 
The distinguished apices are $\dapx(L) = \ldb \vertex L{M_1}, \vertex L{M_2}\rdb = \ldb[0:0:0:0], [0:0:1:1]\rdb$. 
The presentation fan $\phi(M_1)$ consists of two rays, one in direction $e_1$ and the other in direction $e_2$ while $\phi(M_2)$ has its rays going in direction $e_3$ and $e_4$. \Cref{fig_ex1} shows $\phi_L(M_1)$ in blue and $\phi_L(M_2)$ in red. 
The presentation space $\Pi(L)$ consists of the $S_2$ orbit of the product of these fans: in other words, 
\[
\Pi(L) = \phi_L(M_1)\times \phi_L(M_2)\cup \phi_L(M_2)\times \phi_L(M_1).
\]
\end{example}

\begin{example}
The uniform matroid $M= U_{d,n}$ is the unique rank~$d$ matroid such that $\tau_M(\emptyset) = d$. 
The presentation fan of the uniform matroid is an $S_d$-invariant subset of $\TT^{d\times n}$ 
where $(A_1,\dots,A_d)\in \phi(U_{d,n})$ if and only if for every non-empty subset $I\subseteq [d]$,
\[
\left|\bigcap_{i\in I} \relsupp_0(A_i)\right| \leq d-|I|.
\]
The support of the $\{0,\infty\}$-vectors within $\phi(U_{d,n})$ give the set system presentations from \Cref{ex:uniform}.
\end{example}

The reason for calling $\Pi(L)$ a presentation space is the following theorem.
\begin{theorem}
\label{thm_B}
Let $V$ be a transversal valuated matroid. Then $\AAA = \ldb A_1,\dots, A_d\rdb$ is a presentation of $V$ if and only if $(A_1,\dots,A_d)\in\Pi(\Be V)$. 
\end{theorem}

In other words, the theorem asserts that $\Pi(\Be V)\subseteq (\TP)^d$ equals the row-wise projectivization of $\stiefel^{-1}(V)$. 
Notice that if $L = \Be M$ is the Bergman fan of a matroid $M$, then the distinguished set of apices $\dapx(L)$ consists of $\dapx(L) = \ldb \overline e_F : F\in \DF(M)\rdb$. 
So the distinguished set of apices $\dapx(L)$ are the valuated generalization of the unique maximal presentation of a transversal matroid.

We prove the two directions of the equivalence in~\Cref{thm_B} separately.
The easier one is \Cref{prop_halfB}, below. 
The other direction is \Cref{thm_apices}.

\begin{proposition}
\label{prop_halfB}
Let $V$ be a transversal valuated matroid. 
If $\AAA = \ldb A_1,\dots, A_d\rdb$ is a presentation of $V$ then $(A_1,\dots,A_d)\in\Pi(\Be V)$.
\end{proposition}
\begin{proof}
Let $\AAA$ be a presentation of $V$ and let $M\in \dmat(V)$. First assume $M\in \MM(V)$. 
Then by \Cref{prop_zoom_pres} we have that $Z_{\vertex LM}(\AAA)$ is a presentation of $\Be M$. 

By \Cref{prop_A} (\ref{cond2}) there are exactly $\cork(F)$ points in a presentation of $M$ whose relative support with respect to~$0$ contains $F$, for every $F\in \CF(M)$.
By definition of $\tau$ and the M\"obius inversion formula, there are exactly $\tau(F)$ points in $Z_{\vertex LM}(\AAA)$ such that the maximal cyclic flat contained in their relative support with respect to $0$ is $F$, i.e.\ points $x$ such that $\cocl_M(\relsupp_0(x)) = F$. 
Applying this to $F=\emptyset$, we get that
there are exactly $\tau_M(\emptyset)$ points of $Z_{\vertex LM}(\AAA)$ whose relative support with respect to $0$ is an independent set of $M$. 
The tuple formed from the corresponding points in $\AAA$ will then be in $\phi_L(M)$.

Now if $M$ is not in $\MM(V)$ but in $\MM(V/F)$ for some $F\in \CF(\underm V)$, then by \Cref{prop_contractions} there is $\AAA_F\subseteq \AAA$ such that its projection to the $[n]/F$ coordinates is a presentation of $V/F$. 
Then by the same argument as above, there are $\tau_M(\emptyset)$ of those points in $\phi_{\Be{V/F}}(M)$ which proves the desired result as $\iota_F(\phi_{\Be{V/F}}(M)) = \phi_L(M)$. 
\end{proof} 



\subsection{Pseudopresentations}
We recall the following characterization of transversal matroids in the form due to Ingleton \cite{Ingleton}.
Essentially the same characterization, but quantifying over all cyclic sets, was given earlier by Mason \cite{Mason}.
\begin{proposition}
\label{prop_transv_charac}
A matroid $M$ is transversal if and only if for every collection of cyclic flats $F_1,\dots,F_k$ the following inequality is satisfied:
\[
\sum\limits_{\emptyset\neq I \subseteq [k]}(-1)^{|I|}\rk\left(\bigcup\limits_{i\in I}F_i\right)\leq -\rk\left(\bigcap\limits_{i=1}^k F_i\right).
\]
\end{proposition}
Notice that for $k=2$, this is the submodularity axiom of the rank function. 
We also remark that on substituting $\rk(J)=d-\cork(J)$ in the above inequality, the $d$ terms cancel out,
and therefore a formally identical inequality is true where $\rk$ is replaced by $\cork$ and $\leq$ by~$\geq$.
\begin{definition}
Let $M$ be a transversal matroid of rank $d$. We say that a collection $G_1,\dots G_d$ of flats of $M$ is a \newword{pseudopresentation} if 
\[
\ldb\cocl(G_1),\dots\cocl(G_d)\rdb = \DF(M).
\]
\end{definition}
To motivate this definition, note that it is a necessary condition for a presentation of~$M$ 
that the complements of its members be a pseudopresentation (see \Cref{BrualdiDinolt}).

\begin{example}
Consider the uniform matroid $U_{d,n}$ with $d\ge 2$. 
The collection $\ldb \{1\}, \dots, \{1\} \rdb$ consisting of the flat $\{1\}$ with multiplicity $d$ is a pseudopresentation, because $\cocl(\{1\}) = \emptyset$,
matching the computation of $\DF(U_{d,n})$ from \Cref{ex:uniform}.
However, the collection of complements of this collection is not a presentation of $U_{d,n}$ as it fails to meet the conditions of \Cref{BrualdiDinolt}. In particular, the matroid with such presentation would have $1$ as a loop. 
\end{example}

The following lemma says that if a pseudopresentation fails to be the complements of a presentation, then the failure is ``local'', 
that is, there is a distinguished cyclic flat $F$ such that the $G_i$ which extend $F$ were poorly chosen.
In other words, replacing every element in the pseudopresentation which does not extend $F$ by its coclosure does not yield a presentation either.

\begin{lemma}
\label{lemma_local_fail}
Let $M$ be a transversal matroid with $\DF(M) = \ldb F_1,\linebreak[1]\dots,\linebreak[1] F_d \rdb$ and let $G_1,\dots,G_d\in \FF(M)$ be a pseudopresentation. 
Suppose that $G_1,\dots,G_d$ are not the complements of a presentation. Then there exists $F\in \DF(M)$ and $I,J\subseteq[d]$, such that:
\begin{itemize}
	\item $\cocl(G_i) = F$ for every $i\in I$
	\item $F \subsetneq F_j$ for every $j\in J$. 
	\item $\cork\left(\bigcap\limits_{i\in I}G_i \cap \bigcap\limits_{j\in J}F_j\right) < |I|+|J|$
\end{itemize}
\end{lemma}

\begin{proof}
Suppose that  such $F$ does not exist but $G_1,\dots, G_d$ are not the complements of a presentation. Then there is a set of indices $I\subseteq [d]$ such that 
\[
\cork\left(\bigcap \limits_{i\in I}G_i\right) < |I|.
\]

Let $k$ be the number of different elements of $\{\cocl(G_i) : i\in I\}$ and without loss of generality let that set be $\{F_1,\dots,F_k\}$. For $j\in [k]$ let $I_j = \{i\in I : \cocl(G_i) = F_j\}$ and let $m_j = |I_j|$. The $I_j$ clearly partition $I$ so we have that
\[
\sum\limits_{j=1}^k m_j = |I|.
\]
Let $K= \bigcap \limits_{i\in I}G_i$. For any proper subset $J\subseteq [k]$ let 
\[
a_J= \left|K\cap\left(\bigcap \limits_{j\in J}F_j\right)\cap\left(\bigcap\limits_{j\in [k]\setminus J}[n]\setminus F_j\right)\right|
\] 
and let $a_{[k]} = \rk\left(\bigcap \limits_{i\in [k]}F_i\right)$. Notice that for any element $x\in K \setminus \bigcap \limits_{i\in [k]}F_i$, $x$ is a coloop of some $G_i$, so in particular it is a coloop in $K$. Therefore we have that 
\[
\rk(K) = \sum\limits_{J \subseteq [k]} a_J.
\]

Since the $G_i$ are pseudopresentation, we have that $\bigcap\limits_{i\in I_j} G_i$ consists of $F_j$ plus (possibly) some coloops. Since $\bigcap\limits_{i\in I_1} G_i \setminus F_1 \supseteq K\setminus F_1$, we have that
\[
\rk\left(\bigcap\limits_{i\in I_1} G_i\right) \ge \rk(F_1)+ \sum\limits_{J\subseteq [k]\setminus\{1\}}a_J.
\]
As we assume $(F_1, I_1,\emptyset)$ is not a certificate as described in the lemma (as the tuple $(F,I,J)$ in the statement), we have that
\[
m_1 \leq \cork\left(\bigcap\limits_{i\in I_1} G_i\right) \leq \cork(F_1)- \sum\limits_{J\subseteq [k]\setminus\{1\}}a_J.
\]
Now for any $2\leq j \leq k$, let 
\[
J_j = \left\{j'\in [d] : F_j\cup \left(\bigcap\limits_{j''<j} F_{j''}\right)\subseteq F_{j'}\text{ and }F_j\neq F_{j'}\right\}.
\]
By inclusion-exclusion, we have that 
\[
|J_j| \geq \sum_{\emptyset\neq J\subseteq [j-1]} (-1)^{|J|-1}\cork\left(\bigcup\limits_{j'\in J\cup\{j\}} F_{j'}\right).
\]
(The right hand side is counting the number of flats that contain $F_j$ and $F_{j''}$ for some $j''<j$.) Now notice that
\[\left(K\cap\bigcap\limits_{j'<j}F_{j'}\right) \setminus F_j  \subseteq \left(\bigcap\limits_{i\in I_j} G_i \cap \bigcap\limits_{j'\in J_j} F_{j'} \right) \setminus F_j\]
and
\[
\left|\left(K\cap\bigcap\limits_{j'<j}F_{j'}\right) \setminus F_j\right| = \sum\limits_{[j-1] \subseteq J\subseteq [k]\setminus\{j\}}a_J,
\]
so
\[
\rk\left(\bigcap\limits_{i\in I_j} G_i \cap \bigcap\limits_{j'\in J_j} F_{j'} \right) \ge \rk(F_j)+ \sum\limits_{[j-1] \subseteq J\subseteq [k]\setminus\{j\}}a_J.
\]
Similarly as before, we assume the conditions of the lemma are not satisfied for $(F_j,I_j,J_j)$, so
\begin{align*}
m_j &\leq \cork\left(\bigcap\limits_{i\in I_j} G_i \cap \bigcap\limits_{j'=1}^{j-1}F_{j'} \right) -|J_j|\\
&\leq       \left(\cork F_j - \sum\limits_{[j-1] \subseteq J\subseteq [k]\setminus\{j\}}a_J\right) - \sum_{\emptyset\ne  J\subseteq [j-1]} (-1)^{|J|-1}\cork\left(\bigcup\limits_{j'\in J\cup\{j\}} F_{j'}\right) \\
&\leq \sum_{J\subseteq [j-1]} (-1)^{|J|}\cork\left(\bigcup\limits_{j'\in J\cup\{j\}} F_{j'}\right)-\sum\limits_{[j-1] \subseteq J\subseteq [k]\setminus\{j\}}a_J.
\end{align*}
Adding all bounds for the $m_j$ and using \Cref{prop_transv_charac} we get:
\begin{align*}
\sum\limits_{j=1}^{k} m_j = |I| &\leq \sum_{\emptyset\neq J\subseteq [k]} (-1)^{|J|+1}\cork\left(\bigcup\limits_{j\in J} F_{j}\right)-\sum\limits_{J\subsetneq [k]}a_J\\
&\leq \cork(\bigcap\limits_{j=1}^k F_j) -\sum\limits_{J\subsetneq [k]}a_J\\
&= d - \sum\limits_{J\subseteq [k]}a_J\\
&= \cork(K)
\end{align*}
which is a contradiction, as we assumed $|I|>\cork(K)$.
\end{proof}

\begin{example}
Consider $M = U_{1,2}\oplus U_{2,3}$, labelling the ground set so that
$M$ is the sum of the matroid $U_{1,2}$ on $\{1,2\}$ and the matroid $U_{2,3}$ on $\{3,4,5\}$. 
We have $\DF(M) = \ldb 12, 12, 345\rdb$. The collection $\ldb 123,123, 345\rdb$ is a pseudopresentation of $M$, since $\cocl(123) = 12$. 
However, it is not the set of complements of a presentation since they all intersect in $3$ which is not a loop. 
This failure to be a presentation is concentrated in the flats extending $12$, so in terms of \Cref{lemma_local_fail} we have $F = 12$, $G_1 = G_2 = 123$ and $F_3 = 345$. 
\end{example}

\subsection{Paths of points and flats}
The two proofs in \Cref{ssec:path proofs} are both arguments by contradiction
establishing some property of all distinguished flats $F$ of coloopless matroids $M$ indexing a face $\f LM$ in a tropical linear space $L$.
They proceed by reducing a counterexample to another counterexample for different $F$ and $M$.
In this subsection we introduce the reductions used and show that a sequence thereof must terminate.

Let $L = \Be V$ be a tropical linear space such that $V$ has transversal facets.
Let $x\in L$, and let $M$ be the matroid such that $x\in\f LM$. Assume $M$ is coloopless.
Let $F\in\DF(M)$ be a distinguished flat.
Denote by~$H_F$ the supporting hyperplane
\[
H_F := \left\{\sum\limits_{i\in F}z_i = \rk_M(F)\right\}.
\]
If $F=\emptyset$ then $H_F$ is not a hyperplane, but in this event we will not use~$H_F$.

\begin{definition}\label{def:ascendent step}
An \newword{ascendent step} from $(M,F,x)$ is a triple $(M',F',x')$ satisfying conditions given as follows.
\begin{enumerate}\addtocounter{enumi}{-1}
\item\label{as0} If $F=\emptyset$ then there are no ascendent steps.
\item\label{as1} If $F\ne\emptyset$ and $P_M\not\subseteq H_F$ then the ascendent steps are the triples of form
\[(M',F',x')=(M|F\oplus M/F, F, x+\lambda e_F)\]
for some $\lambda>0$ with $x'\in\f L{M'}$.
\item\label{as2} If $F\ne\emptyset$, $P_M\subseteq H_F$ and $\rk_{\underm V}(F) > \rk_M(F)$ then the conditions on an ascendent step $(M',F',x')$ are as follows.
As above, $x'=x+\lambda e_F$, where now $\lambda>0$ is minimal such that $x+\lambda e_F\notin\f LM$.
Then $x'$ is in a cell $\f L{M'}$ which must be a proper face of $\f LM$.
The flat $F'\in\DF(M')$ must be such that $F'\setminus F$ is independent in~$M/F$.
\item\label{as3} If $F\ne\emptyset$, $P_M\subseteq H_F$ and $\rk_{\underm V}(F) = \rk_M(F)$ then there are no ascendent steps.
\end{enumerate}
\end{definition}

We know $\rk_{\underm V}(F)\ge \rk_M(F)$, so these cases are comprehensive.

\begin{definition}
An \newword{ascendent path} is a finite or infinite sequence of triples $(\tau_i)_{i\ge0}$, $\tau_i = (M_i,F_i,x_i)$,
such that for each $i\ge0$,
either $\tau_i$ is the last term of the sequence or $\tau_{i+1}$ is an ascendent step from $\tau_i$.\end{definition}

Let us give some intuition of what an ascendent path is. 
In each ascendent step, we go from the point $x_i$ in a colooples cell $L_{M_i}$ and start going in a straight line within $L$ in direction $e_{F_i}$ until we change the cell of $L$ where we are standing, so long as it is still is coloopless. If that change occurs immediately, that is, $L_{M_{i+1}}$ is of higher dimension than $L_{M_i}$ and $P_{M_{i+1}}$ is a face of $P_{M_i}$, we keep going in the same direction (Case 1). If not, since $L_{M_i}$ is bounded because $M_i$ is coloopless, then that change occurs at a face $L_{M_{i+1}}$ of $L_{M_i}$ i.e. $P_{M_i}$ is a face of $P_{M_{i+1}}$ (Case 2). This is the opposite of the last case in that
\begin{equation}
M_i = M_{i+1}\backslash F_i \oplus M_{i+1}/([n]\setminus F_i).
\label{eq:case2}
\end{equation}

In this case we may choose a new direction, however with the restriction above which is equivalent that $\cocl_{M_i}(F_{i+1}\cup F_i) = F_i$, that is, $F_{i+1}\setminus F_i$ consists of coloops in $M_i|F_{i+1}\cup F_i$. We repeat this until the direction is $\emptyset$ (Case 0) or we leave the bounded region of $L$ (Case 3). 
Again, what we will show for our later uses of this definitions, in \Cref{lemma:terminate},
is that all ascendent paths terminate after finitely many steps (thus for example they cannot loop). 
The reason why we call the paths ``ascendent'' is \Cref{lema:relsupp}.

\begin{lemma}\label{lemma:relative coloops transitive}
Let $((M_i,F_i,x_i))_{i\ge0}$ be an ascendent path.
Then, for any $i\le j$, $\left(\bigcup\limits_{i<k\le j}F_k\right)\setminus F_i$ is independent in $M_i/F_i$.
\end{lemma}

\begin{proof}
We use descending induction on~$i$.  
The base case is $i=j$, where $F_j\setminus F_i$ is empty and therefore independent in any matroid.

If $i<j$ then the ascendent step from $(M_{i+1},F_{i+1},x_{i+1})$ belongs either to Case~1 or Case~2 of Definition~\ref{def:ascendent step}.
In Case~1, $F_{i+1} =F_i$ so 
\[\left(\bigcup\limits_{i<k\le j}F_k\right)\setminus F_i = \left(\bigcup\limits_{i+1<k\le j}F_k\right)\setminus F_{i+1},\]
which by induction hypothesis is independent in
\[M_{i+1}/F_{i+1}=(M_i|F_i\oplus M_i/F_i)/F_{i+1} = M_i/F_i,\]
which is what is needed.

In Case~2, first notice that
\[
M_i / F_i = (M_{i+1}\setminus F_i \oplus M_{i+1}/([n]\setminus F_i))/F_i = M_{i+1}\backslash F_i.
\]
By definition of ascendent step, $F_{i+1} \setminus F_i$ is independent in $M_i / F_i = M_{i+1}\backslash F_i$, 
so it is also independent in any restriction of $M_{i+1}$ that contains it, in particular in $M_{i+1}|(U\cup F_{i+1})$.
By the induction hypothesis, $U =  \left(\bigcup\limits_{i+1<k\le j}F_k\right)\setminus F_{i+1}$ consists of coloops of $M_{i+1}|(U\cup F_{i+1})$, 
so the set $U\cup (F_{i+1} \setminus F_i)$, being obtained by adding coloops to $F_{i+1} \setminus F_i$, is also independent in $M_{i+1}|(U\cup F_{i+1})$. But
\[
\left(\bigcup\limits_{i<k\le j}F_k\right)\setminus F_i\subseteq U\cup (F_{i+1} \setminus F_i),
\]
so $\left(\bigcup\limits_{i<k\le j}F_k\right)\setminus F_i$ is also independent in $M_{i+1}|(U\cup F_{i+1})$ and hence in  $M_i / F_i = M_{i+1}\backslash F_i$. 
\end{proof}

\begin{lemma}
\label{lema:relsupp}
Let $((M_i,F_i,x_i))_{i\ge0}$ be an infinite ascendent path. Then the sequence of $(x_i)_{i\ge 0}$ is nondecreasing when written in supportive coordinates with respect to~$x_0$, 
i.e.\ for every $\ell$, the $\ell$-th coordinate of $x_i$ is a nondecreasing function of~$i$.
\end{lemma}

\begin{proof}

For each $j\ge0$ we have that $x_{j+1}=x_j+\lambda_je_{F_j}$ and thus
\begin{equation}\label{eq:marelsupp}
x_j = x_0+\sum_{0\le i<j}\lambda_ie_{F_i}
\end{equation}
in $\mathbb R^n/\mathbb R(1,\ldots,1)$, for positive reals $\lambda_i$.
Fix a coordinate vector for~$x_0$.
The lemma is immediate
once we show that $\bigcup_{0\le i<j}F_i$ is not the whole ground set $[n]$,
as this implies that \eqref{eq:marelsupp} remains true when $x_j$ is given supportive coordinates with respect to~$x_0$,
with $(x_0)_a=(x_j)_a$ for any $a\not\in\bigcup_{0\le i<j}F_i$.
But this follows from \Cref{lemma:relative coloops transitive}. 
Indeed, $\left(\bigcup_{0<k\le j}F_k\right)\setminus F_0$ cannot equal $[n]\setminus F_0$
because if $[n]\setminus F_0$ were independent in $M_0/F_0$ it would consist entirely of coloops in $M_0$,
but $M_0$ was assumed coloop-free.
\end{proof}

\begin{lemma}
\label{lemma:terminate}\label{lemma:path technique}
Infinite ascendent paths do not exist.

In particular, if $S$ is a set of triples $(M,F,x)$ such that for all $\tau\in S$,
an ascendent step from $\tau$ is also in $S$, then $S$ is empty.
\end{lemma}

\begin{proof}
We argue that if $(M, x, F)$ is followed by an ascendent step of Case~1 in an ascendent path, 
then $M$ can never appear subsequently in the path. This proves the result,
because $V$ has only finitely many initial matroids,
and every step in Case~2 decreases the number of connected components of~$M$ so an infinite sequence of just Case~2 steps can't occur either.

By the assumption $P_M\not\subseteq H_F$ of Case~1, there exists $B\in \BB(M)$ such that $|B\cap F| <\rk_M(F)$.
Suppose that $(M,y,G)$ appears subsequently in the path. We have $M = V^y$, so $B \in \BB(V^y)$. 
By \Cref{lema:relsupp}, $\relsupp_x(y)$ consists of $F$ plus possibly some other elements
which by \Cref{lemma:relative coloops transitive} are coloops of $M|\relsupp_x(y)$. 
As $B\cap F$ is an independent set in $M|\relsupp_x(y)$, we can extend it to a basis $\tilde B$ of $\tilde M=M|\relsupp_x(y)$.
Since $F\subseteq\relsupp_x(y)$, we can arrange that $\tilde B$ contains $\rk_{\tilde M}(F) = \rk_M(F)$ elements of~$F$.
Also, $\tilde B$ contains all of the coloops of~$\tilde M$.
Extend further to a basis $B'$ of~$M$ containing $\tilde B$.
Since $B$ contains fewer than $\rk_M(F)$ elements of~$F$, this construction arranges that 
$B'\cap \relsupp_x(y)$ is a strict superset of $B\cap \relsupp_x(y)$. 
By definition of relative support, this containment implies
\begin{equation}\label{eq:terminate1}
\sum_{i\in B}(y_i-x_i) < \sum_{i\in B'}(y_i-x_i).
\end{equation}

Since $M = V^x$, we have that the basis $A\in \BB(\underm V)$ causes $V_A - \sum_{i\in A}x_i$ to take its minimum value exactly when $A\in\BB(M)$. 
In particular $V_B - \sum_{i\in B}x_i = V_{B'} - \sum_{i\in B'}x_i$.
Subtracting \eqref{eq:terminate1} gives $V_B - \sum_{i\in B}y_i > V_{B'} - \sum_{i\in B'}y_i$, so $B$ cannot be a basis of $V^y$, a contradiction. 

%
The final claim is clear.
\end{proof}

\subsection{Proof of \Cref{thm_B}}\label{ssec:path proofs}
Throughout this subsection,
let $L = \Be V$ be a tropical linear space such that $V$ has transversal facets  
and let $\dapx(L) = \ldb A_1,\dots,A_d\rdb$ be its distinguished multiset of apices. 
Let $(A_1',\dots, A_d') \in \Pi(L)$ and $\AAA'=\ldb A_1',\dots, A_d'\rdb$ be such that, for each $i$, $A_i'$ lies in the cone with apex $A_i$.

\begin{proposition}
\label{prop_pseudo}
Under the assumptions of this subsection,
for every $M\in\MM(V)$ coloop-free and $x\in \f LM$, the multiset $\ldb \relsupp_x(A'_1),\dots,\relsupp_x(A'_d)\rdb$ is a pseudopresentation of $M$. 
\end{proposition}

\begin{proof}
Consider such $x\in \f L M$ for a coloopless $M$. If $\ldb \relsupp_x(A'_1),\linebreak[1]\dots,\linebreak[1]\relsupp_x(A'_d)\rdb$ is not a pseudopresentation then there is a flat $F\in \DF(M)$ such that
\begin{equation}
|\{i : \cocl_M(\relsupp_x(A'_i)) = F\}| < \tau_M(F).
\label{eq:pseudofail}
\end{equation}
We show that there is no such triple $(M,F,x)$ using \Cref{lemma:path technique},
by either showing a contradiction directly or constructing an ascendent step
that also satisfies \eqref{eq:pseudofail}.
The proof is arranged according to the cases of \Cref{def:ascendent step}.

\noindent\textbf{\textit{Case 0.}}
If $F=\emptyset$ then $\tau_M(\emptyset)> 0$ so $M\in \dmat(L)$ and the multiplicity of $x$ in $\dapx(L)$ is exactly $\tau_M(\emptyset)$. 
If $A_i$ is a distinguished apex with $\relsupp_x(A_i) = \emptyset$ then by \Cref{def:presentation_fan} applied to~$\phi_M$, $\cocl_M(\relsupp_x(A_i')) = \emptyset$. 
So
\[
|\{i : \cocl_M(\relsupp_x(A'_i)) =\emptyset\}| \ge
|\{i : \cocl_M(\relsupp_x(A_i)) =\emptyset\}| \ge \tau_M(\emptyset).
\]

\noindent\textbf{\textit{Case 1.}}
Let $(M',F,x')$ be the ascendent step from $(M,F,x)$. 
The lattice of flats of $M'$ decomposes as
$(\FF(M'), {\subset}) = (\FF(M|F), {\subset})\times(\FF(M/F), {\subset})$ 
where $(\FF(M|F), {\subset})$ is isomorphic to the sublattice of $(\FF(M'), {\subset})$ below $F$ 
and $(\FF(M/F), {\subset})$ is isomorphic to the sublattice above $F$. 
In particular, $\tau_{M'}(F) = \tau_{M}(F)$. 
If there are $\tau_{M'}(F) = t$ points 
$A'_i$ with $\cocl_{M'}(\relsupp_{x'}(A'_i)) = F$, 
then those same apices satisfy $\relsupp_x(A'_i) = \relsupp_{x'}(A'_i)$ and $\cocl_M(\relsupp_x(A'_i)) = \cocl_{M'}(\relsupp_{x'}(A'_i)) = F$.
So \Cref{eq:pseudofail} for  $(M,F,x)$ implies \Cref{eq:pseudofail} for $(M',F,x')$.

\noindent\textbf{\textit{Case 2.}}
Recall that in this case $M$ and $M'$ are related by \Cref{eq:case2}.
If $\rk_M(F)=r$, then $\DF(M)$ contains exactly $r$ supersets (possibly not strict) of $[n]\setminus F$,
which will also be in $\DF(M')$ because 
the upper intervals above $[n]\setminus F$ are identical in $\FF(M)$ and~$\FF(M')$.
For $F'\in\CF(M')$ a proper subset of $[n]\setminus F$, we have that
\begin{equation}\label{eq:tau case 2'}
\tau_M(F\cup F') = |\ldb G\in\DF(M') : F' = \cocl_{M'}(G\setminus F) \rdb|.
\end{equation}
To see this, compare the use of the recursion \eqref{eq:definition of tau} to compute
$\tau_M$ on the interval $[F, [n]]$
and $\tau_M'$ on the interval $[\emptyset, [n]\setminus F]$.
Note that these two intervals are isomorphic.
The coranks in the latter interval exceed those in the former by~$r$;
this is accounted for by the $r$ distinguished flats of~$M'$ above  $[n]\setminus F$.
The other difference is the presence of flats $G$ not comparable with $[n]\setminus F$ in $M'$.
Because $\CF(M')$ is a lattice, 
it contains a greatest lower bound of $G$ and $[n]\setminus F$,
namely $\cocl_{M'}(G\setminus F)$.
This is the maximal element of $[\emptyset, [n]\setminus F]$ contained in~$G$.
Therefore, terms $\tau(G)$ behave in the recursion as if they were terms $\tau(\cocl_{M'}(G\setminus F))$, and this is the fact expressed by \eqref{eq:tau case 2'}.

The case $F'=\emptyset$ of \Cref{eq:tau case 2'} means that if $t= \tau_{M}(F)$ there are exactly $t$ elements $\ldb F_1,\dots, F_t \rdb \subseteq \DF(M')$ such that $F_i\setminus F$ is an independent set in~$M'$. 
In particular, $\cocl_M(F_i\cup F) =F$ for every $i\in [t]$. 
Then any point 
$A'_i$ that satisfies $\cocl_{M'}(\relsupp_{x'}(A'_i)) = F_i$ must satisfy $\cocl_M(\relsupp_x(A'_i)) = \cocl_M(F_i\cup F) = F$. 
So again, \Cref{eq:pseudofail} for $(M,F,x)$ implies that there is an $F'$ such that \Cref{eq:pseudofail} holds for the ascendent step $(M',F',x')$.

\noindent\textbf{\textit{Case 3.}}
In this case $P_M$ is in the boundary of $P_{\underm V}$ and the affine span of $\f LM$ contains $e_F$. 
In particular $\f LM$ is unbounded in the $e_F$ direction. But then $M' = M/F$ is a coloopless matroid with $\tau_{M'}(\emptyset)=\tau_M(F)> 0$, so $M'$ is connected and $\f L{M'}$ consists of just a vertex $v$ with infinity in the coordinates corresponding to $F$. 
In particular, the multiplicity of $v$ in $\dapx$ is $\tau_M(F)$,
i.e.\ $\relsupp_x(A_i) = F$ holds for $\tau_M(F)$ values of~$i$.
By \Cref{def:presentation_fan} applied to~$\phi_{M/F}$, $\cocl_M(\relsupp_x(A_i')) = F$ for each of these~$i$, 
so \Cref{eq:pseudofail} cannot hold.	
\end{proof}

\begin{theorem}
\label{thm_apices}
Under the assumptions of this subsection, $\AAA'$ is a presentation of~$L$.
\end{theorem}
\begin{proof}
If 
$\AAA' = \ldb A'_1,\dots,A'_d\rdb$ is not a presentation of $L$, 
then by \Cref{theorem_A}
there exists $x\in \f LM$ where $M$ is a coloopless matroid such that $\relsupp_x(\AAA') = \ldb \relsupp_x(A'_1),\linebreak[1]\dots,\linebreak[1]\relsupp_x(A'_d)\rdb$ is not a presentation. 
By \Cref{prop_pseudo}, $\relsupp_x(\AAA')$ is indeed a pseudopresentation, 
so by \Cref{lemma_local_fail} we know there is a flat $F\in \DF(M)$, 
a set $I$ such that $\cocl_{M}(\relsupp_x(A'_i)) = F$ for every $i\in I$ 
and distinguished flats $\ldb F_1,\dots,F_k\rdb \subseteq \DF(M)$ such that $F_j \supsetneq F$ for every $j\in[k]$ and 
\[
\cork_{M}\left(\bigcap\limits_{i\in I}\relsupp_x(A'_i) \cap \bigcap\limits_{j= 1}^kF_j\right) < |I|+k.
\]
We now use \Cref{lemma:path technique},
either directly showing a contradiction
or constructing an ascendent step $(M',F',x')$ from $(M,F,x)$
that exhibits the same failure of presentation.
Again, we break into the cases of \Cref{def:ascendent step}.

\noindent\textbf{\textit{Case 0.}}
If $F=\emptyset$, contradiction is immediate because $\cork_M(\emptyset) = d$.

\noindent\textbf{\textit{Case 1.}}
For $x'= x+\lambda e_F$ with small enough $\lambda$, we have that $\relsupp_{x'}(A'_i) =\relsupp_x(A'_i)$ for any $i \in I$. Since for any set $S$ that contains $F$ we have that $\cork_M(S) = \cork_{M'}(S)$, we conclude that
\[
\cork_{M'}\left(\bigcap\limits_{i\in I}\relsupp_{x'}(A'_i) \cap \bigcap\limits_{j= 1}^kF_j\right) < |I|+k.
\]		 

\noindent\textbf{\textit{Case 2.}}
Here, $\relsupp_x(A'_i) \supseteq \relsupp_{x'}(A'_i) \supseteq \relsupp_x(A'_i) \setminus F$. For every $j\in [k]$, $F_j\setminus F$ is a cyclic flat in $M'$. However, it may be the case that $\tau_{M'}(F_j\setminus F) \leq \tau_{M}(F_j)$. This happens when there is a cyclic flat $F_j'$ such that $F_j' \setminus F = F_j \setminus F$. In any case, we can find distinguished flats $\ldb F_1',\dots,F_k'\rdb \subseteq \DF(M')$ such that for every $j\in[k]$ we have $F_j' \setminus F = F_j \setminus F$. 
Moreover, there are another $r = \cork_{M'}([n]\setminus F) = \rk_M(F)$ distinguished flats $F_{k+1}',\dots, F_{k+r}'$
such that $F_{k+j}'\supseteq [n]\setminus F$ for every $j\in[r]$. In total we have that
\begin{align*}
\bigcap\limits_{i\in I}\relsupp_{x'}(A'_i) \cap \bigcap\limits_{j= 1}^{k+r}F_j' &\supseteq \left(\bigcap\limits_{i\in I}\relsupp_x(A'_i) \cap \bigcap\limits_{j= 1}^kF_j\right)\setminus F\displaybreak[1]\\
\rk_{M'}\left(\bigcap\limits_{i\in I}\relsupp_{x'}(A'_i) \cap \bigcap\limits_{j= 1}^{k+r}F_j'\right) &\ge \rk_{M'}\left(\left(\bigcap\limits_{i\in I}\relsupp_x(A'_i) \cap \bigcap\limits_{j= 1}^kF_j\right)\setminus F\right)\\
&\ge rk_M\left(\left(\bigcap\limits_{i\in I}\relsupp_x(A'_i) \cap \bigcap\limits_{j= 1}^kF_j\right)\setminus F\right)\\
&\ge rk_M\left(\bigcap\limits_{i\in I}\relsupp_x(A'_i) \cap \bigcap\limits_{j= 1}^kF_j\right) -\rk_M(F)\displaybreak[1]\\
\cork_{M'}\left(\bigcap\limits_{i\in I}\relsupp_{x'}(A'_i) \cap \bigcap\limits_{j= 1}^{k+r}F_j'\right) &\leq \cork_M\left(\bigcap\limits_{i\in I}\relsupp_x(A'_i) \cap \bigcap\limits_{j= 1}^kF_j\right) +\rk_M(F)\\
&<|I|+k+r.
\end{align*}
Then 
$\relsupp_{x'}(\ldb A'_1,\ldots,A'_d\rdb)$ is not a presentation of $M$. 
So we can use \Cref{lemma_local_fail} again to find $F'\in \DF(M')$ and $I'$ such that\linebreak[4] $\cocl_{M'}(\relsupp_{x'}(A'_i)) = F'$ where the conditions for presentation fail. 

The only thing left to prove is that $F'\setminus F$ is independent in $M/F$, for $(M',F',x')$ to be indeed an ascendent step from $(M,F,x)$. 
Notice that it follows from the proof of \Cref{lemma_local_fail} that $I'\subseteq I$. Then for any $i\in I'$ we have that $F'\subseteq \relsupp_{x'}(A'_i) \subseteq \relsupp_x(A'_i)$. As $\cocl_M(\relsupp_x(A'_i)) = F$, then $\cocl_M(F')\subseteq F$ where it follows that $F'\setminus F$ is independent in $M/F$.

\noindent\textbf{\textit{Case 3.}}
Notice that 
\[\DF(M/F) = \ldb F_i\setminus F : F_i\in \DF(M), \, F\subseteq F_i \rdb. 
\]
Then
\[
\cork_{M}\left(\relsupp_x(A_i') \cap \bigcap\limits_{j= 1}^kF_j\right) < |I|+k
\]
implies 
\[
\cork_{M/F}\left((\relsupp_x(A_i')\setminus F) \cap \bigcap\limits_{j= 1}^k(F_j\setminus F)\right) < |I|+k.
\]
But this is a contradiction to the definition of $\phi_{M/F}$, which says that there is a presentation of $M/F$ containing 
\[
\ldb [n] \setminus \relsupp_x(A_i') \mid i \in I\rdb \cup \ldb [n] \setminus F_j \mid j\in[k] \rdb.\qedhere
\]
\end{proof}

\subsection{Further consequences}\label{ssec:path consequences}

A corollary of the above results is the converse of \Cref{prop:FR local}.
\begin{theorem}
\label{coro_transv}
A tropical linear space is in the Stiefel image if and only if all the facets in its dual subdivision are transversal.
\end{theorem}
Since the class of transversal matroids is closed under
contractions of cyclic sets \cite[Theorem 5.4]{BrualdiMason} and arbitrary deletions,
if $V$ is transversal then so is any initial matroid $V^x$ which has no new coloops.
Thus \Cref{coro_transv} can be sloganized: 
\emph{transversality is a local property of a tropical linear space}.

\begin{corollary}
Let $M$ be a matroid and suppose $P_M$ has a regular subdivision such that all facets in the subdivision are transversal. Then $M$ is transversal.
\end{corollary}
\begin{proof}
Let $L$ be a tropical linear space dual to such a regular subdivision. 
By \Cref{coro_transv}, $L$ is in the Stiefel image so it has a presentation $A$. 
Consider the matrix $\tilde{A}$ that replaces all finite entries of $A$ by 0. 
Then $\stiefel(\tilde{A})$ is the Bergman fan of $M$, so $M$ is transversal. 
\end{proof}

\begin{figure}[phtb]
	\centering
		\includegraphics[scale = 0.45]{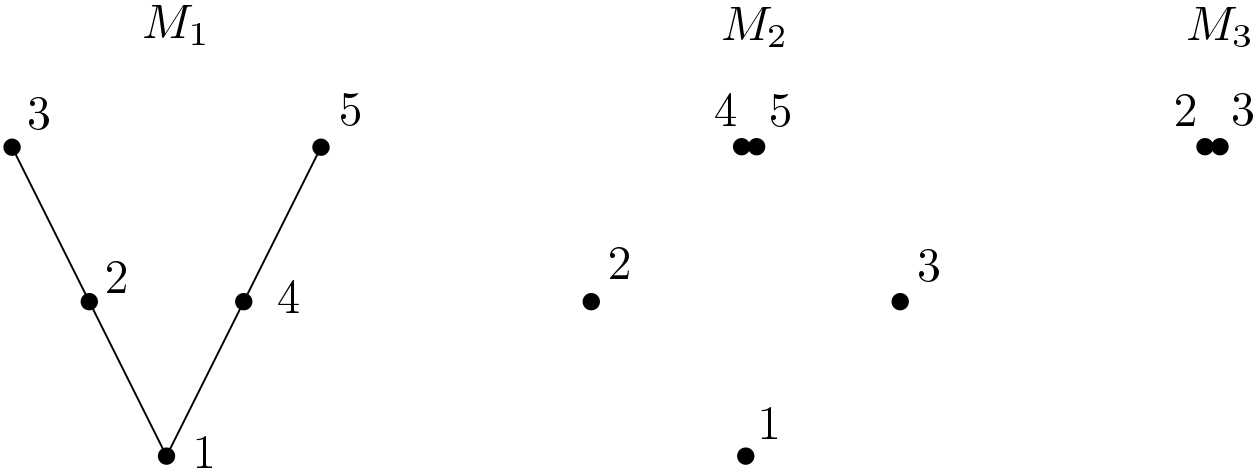}
	\caption{The distinguished matroids of $V$ in \Cref{ex2}.}
	\label{fig_dmat}
\end{figure}
\begin{figure}[phtb]
	\centering
		\includegraphics[scale = 0.5]{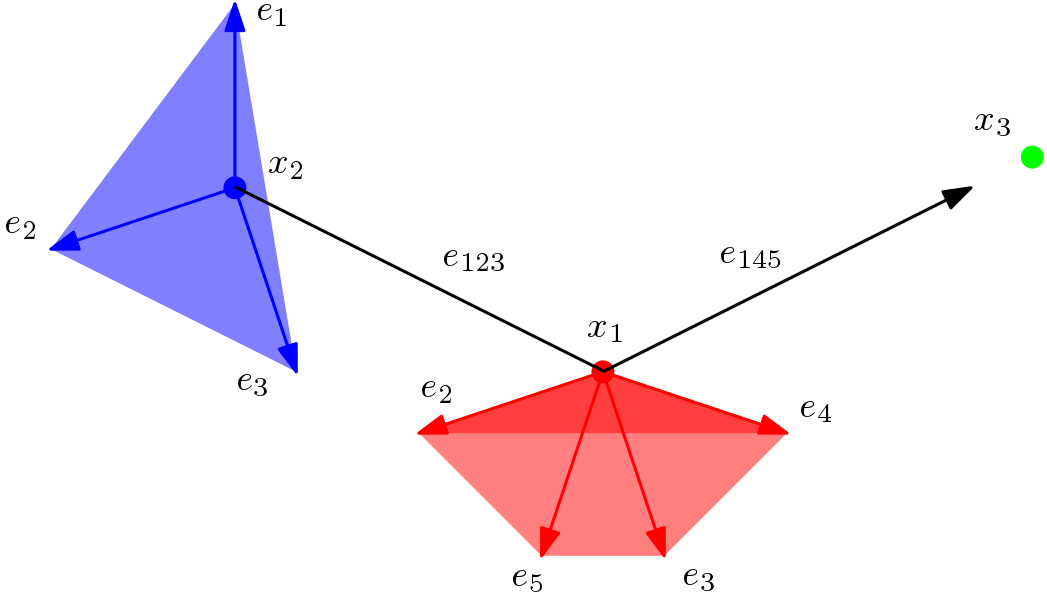}
	\caption{The presentation fan $\phi_{M_i}$ of each of the distinguished matroids $M_i$ in \Cref{ex2},
	as they appear together in~$\mathbb{TP}^{4}$.
	Labels $e_J$ on rays and edges indicate their directions.}
		\label{fig_ex2}
\end{figure}
\begin{example}
\label{ex2}
Let $V$ be the valuated matroid of \Cref{ex:zoom}:
we recall that $V$ was of rank $3$ on 5 elements such that $V_{123}= 1$, $V_{145}= \infty$, and $V_B=0$ for any $B\in\binom{[5]}3$ other than these two. 
The three distinguished matroids $M_1$, $M_2$ and $M_3$ of~$V$ are shown in \Cref{fig_dmat}. 
The respective distinguished apices of $\Be V$ are $x_1 = [0:0:0:0:0]$, $x_2=[1:1:1:0:0]$ and $x_3 = [\infty:0:0:\infty:\infty]$. 
\Cref{fig_ex2} shows the presentation fan of each distinguished matroid: 
the fan from~$x_1$ is the cone over the boundary of a square and
the fan from~$x_2$ is the cone over the boundary of a triangle,
while the fan from~$x_3$ is the single point $x_3$.
So any matrix $A\in \stiefel^{-1}(V)$ must have one row in the red zone, another row in the blue zone and a third row lying exactly at the green point. 
\end{example}

\section{Strict gammoids and stable intersection}\label{sec:stable}
The first appearance of \newword{stable intersection} of tropical varieties
was as the \emph{fan displacement rule} of Fulton and Sturmfels \cite{FultonSturmfels}. 
Speyer \cite[Section 3]{Speyer} described the special case of stable intersection for tropical linear spaces in terms of Pl\"ucker coordinates.
\begin{definition}
Let $V$ and $V'$ be valuated matroids on~$[n]$ of respective ranks $d$ and~$d'$.
Their stable intersection $V\stableint V'$ is the valuated matroid of rank $d+d'-n$ defined by
\[\pl{(V\stableint V')}J = \min\{\pl VB+\pl{V'}{B'} : \textstyle B\in\binom{[n]}d,B'\in\binom{[n]}{d'},B\cap B'=J\}\]
for each $J\in\binom{[n]}{d+d'-n}$,
provided that there exists some $J$ for which the above formula yields $\pl{(V\stableint V')}J<\infty$. 
\end{definition}
In particular, for such a valuated matroid to exist we must have $d+d' \ge n$. 
By comparing this definition to \Cref{rem:stable sum}, we see that stable intersection is dual to stable sum, in the sense that
\[
(V \stableint V')^* = V^*+V'^* \quad\text{ and }\quad (V + V')^* = V^*\stableint V'^*.
\]
The linear space $L(V\stableint V')$ is contained inside $L(V)\cap L(V')$ but in general this containment can be strict (for example, whenever $V= V'$). 

\medskip
In matroid theory, the dual of a transversal matroid is commonly known as a \newword{strict gammoid}.
\begin{definition}
Let $\Gamma =([n], E)$ be a directed graph with vertices $[n]$ and directed edges $E\subset [n]^2$, 
and let $J\subseteq [n]$ be a subset of size $d$. 
A \newword{linking} from a set $B\subseteq [n]$ to~$J$ is a collection of vertex-disjoint directed paths
such that each path starts from a vertex in $N$ and ends in~$J$,
and each vertex of~$B$ is the start of exactly one path. 
\end{definition}
We allow a path to be zero edges long.
\begin{proposition}\label{prop:strict gammoid}
The collection of all sets $B$ of size $d$ such that there is a linking from $B$ to $J$ 
is the set of bases of a matroid. 
A matroid arises this way if and only if it is the dual of a transversal matroid. 
\end{proposition}
The first sentence of \Cref{prop:strict gammoid} is due to Mason~\cite{MasonGammoids}, the second to Ingleton and Piff~\cite{IngletonPiff}.

Our work provides a valuated version of strict gammoids. 
We now describe these in terms of weighted directed graphs, akin to the graphs Speyer and Williams use to parametrize the tropical positive Grassmannian \cite{SpeyerWilliams}.
Consider a weighted directed graph $\Gamma = ([n], E)$ with vertices $[n]$ and where $E$ is now a weight function $E: [n]^2 \to \TT$ which is $0$ on the diagonal. The directed edges of the graph are exactly the pairs where $E$ takes finite value. Let $J\subseteq [n]$ be a subset of size $d$. 
Given a linking from a set $B$ to~$J$, the \newword{weight} of that linking is the sum of the weights of all of the edges used in that linking.
\begin{proposition}
\label{prop:dual}
Let $\Gamma$ be a weighted directed graph with no negative cycles. Let $V\in \TT\PP^{\binom{n}{d} -1}$ be the vector such that for every subset $B\in {[n]\choose d}$, $V_B$ is the minimum weight among all linkings from $B$ to $J$. Then $V$ is a valuated matroid.
Moreover, a valuated matroid arises this way if and only if it is the dual of a transversal valuated matroid.
\end{proposition}
We call any such valuated matroid a \emph{valuated strict gammoid}.
\begin{proof}
Consider $A\in \TT^{(n-d)\times n}$ to be the matrix where the rows are indexed by $I = [n]\setminus J$ and $A_{i,j}$ is the weight of the edge from $i$ to $j$. In particular, $A_{i,i}$ is $0$ for every $i\in I$.  
Let $B\in {[n]\choose d}$ and consider the tropical minor of $A$ corresponding to the columns $[n]\setminus B$. 
A matching from those columns to the rows corresponds to picking edges such that every vertex in $[n]\setminus B$ has exactly one edge coming in and all vertices in $I$ have exactly one edge coming out. 
Taken together this is exactly a linking from $B$ to $J$ plus possibly some cycles in $I\setminus B$. The value of the term of that matching in the corresponding tropical minor
is equal to the weight of the linking plus the weights of the cycles. However, as there are no negative cycles, removing the cycles (choosing the matching where for every vertex $i$ in a cycle is matched with itself instead) the value of the corresponding term can only decrease. So the corresponding minor is equal to the minimum weight of a matching for $B$ to $J$, that is, $V_B$. 
This shows $V$ is exactly the dual of~$\stiefel(A)$.

Now if $V$ is dual to a transversal valuated matroid $\stiefel(A)$ with $A\in \TT^{(n-d)\times n}$, to construct the corresponding weighted graph $\Gamma$, let $I$ be any basis of $\underm{\stiefel(A)}$ and let $\sigma: [n-d] \to I$ be a matching that achieves the minimum of $\stiefel(A)_I$. Let $\Gamma$ be the weighted directed graph where for every $(i,j)\in I \times[n]$ there is an edge from $i$ to $j$ with weight $A_{\sigma^{-1}(i),j}-A_{\sigma^{-1}(i),i}$. 
As $\sigma$ achieves the minimum among matchings $[n-d] \to I$ there cannot be any negative cycles in $\Gamma$. 
So when the matrix $A'$ is constructed from $\Gamma$ as described above, 
then $A'$ is obtained from $A$ by subtracting $A_{\sigma^{-1}(i),\sigma(i)}$ 
from each entry of the row $\sigma^{-1}(i)$. 
In particular $\stiefel(A') = \stiefel(A)$, so $V$ is the valuated matroid associated to $\Gamma$.
\end{proof}

As a corollary from \Cref{coro_transv} and \Cref{prop:dual} we get the following.

\begin{theorem}
\label{thm:gammoids}
Let $V$ be a valuated matroid. Then the following are equivalent:
\begin{enumerate}
	\item $V$ is a valuated strict gammoid.
	\item $\Be V$ is the stable intersection of tropical hyperplanes.
	\item Every connected matroid in $\MM(V)$ is a strict gammoid.
\end{enumerate}
\end{theorem}
Furthermore, \Cref{thm_B} explicitly describes the space of all $d$-tuples of tropical hyperplanes whose stable intersection is $\Be V$ and, 
through \Cref{prop:dual}, all possible weighted directed graphs $\Gamma$ representing $V$ as a valuated strict gammoid.

\begin{example}
Recall the snowflake tropical linear space $L = \Be V$ from \Cref{ex:no transversal}. 
As we said, $V$ is not a transversal valuated matroid; however, its dual is. Indeed, the following are all the connected matroids in $\MM(V^*)$:
\begin{align*}
\BB(M_1) &= \binom{[6]}{4}  \setminus \{1234,1256,3456\} \quad & \vertex L{M_1}=: x_1 = [0:0:0:0:0:0]\\
\BB(M_2) &= \left\{B \in \binom{[6]}{4} : 56 \not\subset B \right\} \quad & \vertex L{M_2}=: x_2 = [1:1:1:1:0:0]\\
\BB(M_3) &= \left\{B \in \binom{[6]}{4} : 34 \not\subset B \right\} \quad & \vertex L{M_3}=: x_3 = [1:1:0:0:1:1]\\
\BB(M_4) &= \left\{B \in \binom{[6]}{4} : 12 \not\subset B \right\} \quad & \vertex L{M_4}=: x_4 = [0:0:1:1:1:1]
\end{align*}
All of these are transversal. We have that 
\[\tau_{M_1}(\emptyset) =\tau_{M_2}(\emptyset) =\tau_{M_3}(\emptyset) =\tau_{M_4}(\emptyset) =1,\] 
so
\[\dmat(V^*) = \ldb M_1, M_2, M_3, M_4\rdb \text{ and } \dapx(\Be {V^*}) = \ldb x_1,x_2, x_3, x_4\rdb.\]
The presentation fan of $M_i$ is 3-dimensional for each $i$. 
For $J\in \binom{[6]}{3}$, let $[0,\infty]^J \subseteq \TT\PP^5$ be the closed cone containing the points $x$ such that
$x_j \in [0,\infty]$ for $j\in J$ and $x_j = 0$ for $j\notin J$. The presentation fans are:
\begin{align*}
\phi(M_1) &= \bigcup_{J\in \binom{[6]}{3} \enspace J\not \subset \{1234\} \atop J\not \subset \{1256\} \enspace J\not \subset \{3456\} } [0,\infty]^J &\phi(M_2) = \bigcup_{J\in \binom{\{1234\}}{3}} [0,\infty]^J \\
\phi(M_3) &= \bigcup_{J\in \binom{\{1256\}}{3}} [0,\infty]^J &\phi(M_4) = \bigcup_{J\in \binom{\{3456\}}{3}} [0,\infty]^J 
\end{align*}
So any presentation $\present =\ldb A_1,A_2,A_3,A_4 \rdb$ of $V^*$ is of the form $A_i = x_i +a_i$ with $a_i\in \phi(M_i)$. 
Thus the snowflake $L$ is the stable intersection of the four tropical hyperplanes $H_i$ with apex $A_i$ for any such presentation. For example, the rows of the matrix
\[
A = \begin{pmatrix}
	0 & \infty & 0 & \infty & 0 & \infty \\
	\infty & \infty & \infty & 1 & 0 & 0 \\
	\infty & \infty & 0 & 0  & 1  & \infty \\
	0 & 0 & 1 & \infty  & \infty  & \infty 
\end{pmatrix}
\] 
form a presentation of $V^*$. From this presentation, together with the matching $\sigma(1) = 1$, $\sigma(2) = 5$. $\sigma(3) = 3$ and $\sigma(4) = 2$ (as in the proof of \Cref{prop:dual}), we obtain the weighted directed graph from \Cref{fig:gammoid} representing $V$.

\begin{figure}
	\centering
		\includegraphics[width = 64mm]{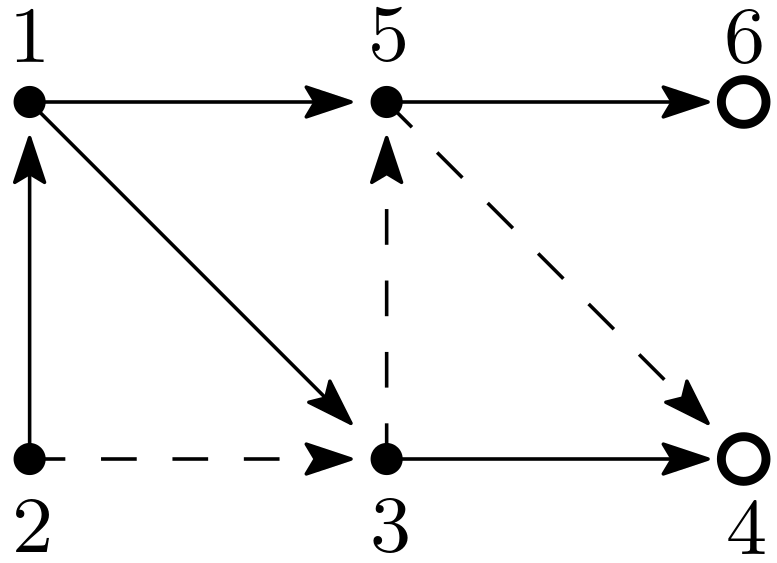}
	\caption{A weighted directed graph representing the snowflake as a valuated strict gammoid. The sinks are $4$ and $6$, the dashed arrows are of weight 1 and all other arrows are of weight 0.}
	\label{fig:gammoid}
\end{figure}

\end{example} 

Notice that given a valuated strict gammoid $V$, collections of tropical hyperplanes whose stable intersection is $\Be V$ together with a matching $\sigma$ are in bijection with weighted directed graph representations of $V$. 

\section{Other connections}
%

\subsection{Gammoids and maps}
Among matroids, the class of \newword{gammoids} is the minor-closure of either of the classes of valuated matroids or strict gammoids.
So a class of valuated gammoids could be defined either as contractions of the transversal valuated matroids that are our main subject
or as restrictions of the valuated strict gammoids of \Cref{sec:stable}.
Valuated gammoids are exactly the images of morphisms from free matroids in the sense of Frenk \cite[\S4.2]{Frenk},
whose results are essentially a tropical formulation of earlier results from \cite{KobayashiMurota,Mitrofanov,Murota-bimatroids}.

\subsection{Tropical convexity}
As explained in Section~\ref{sec:intro}, the tropical Stiefel map is one tropical counterpart
of the map from a matrix to its rowspace.
A different counterpart is the set of all $\TT$-linear combinations of a set of tropical vectors.
This is known as the \newword{tropical cone}.
If the coefficients in the $\TT$-linear combination are further restricted to sum to~$0$ (the multiplicative identity element),
we get the \newword{tropical convex hull}.
Tropical cones and convex hulls have been intensely studied from many points of view
\cite{ArdilaDevelin,DS,JoswigLoho,AGG,GK,Sergeev,Butkovich}.

Tropical cones are usually not tropical linear spaces at all:
\cite[Theorem 16]{YuYuster} describes when they are.
However, tropical linear spaces are tropically convex \cite[Theorem~7]{DS}.
\Cref{lem:containment} implies the following.
\begin{corollary}[{\cite[Theorem~6.3]{FR}}]\label{cor:convex hull contained}
The Stiefel tropical linear space $\Be{\stiefel (A)}$
contains the tropical cone $\TT^d\cdot A$.
\end{corollary}

Thus, the tropical Stiefel map provides a bridge between these two objects, by giving a tropical linear space containing a given tropical cone
(\Cref{cor:convex hull contained}).
If the tropical cone is $r$-dimensional and defined by $r+1$ points,
then the tropical Stiefel map provides an $r$-dimensional tropical linear space,
which is smallest possible.


Every bounded cell of $\Be{\stiefel (A)}$ is contained in the tropical cone $\TT^d\cdot A$ \cite[Theorem~6.8]{FR}.
More generally, $\TT^d\cdot A$ contains the cells of $\Be{\stiefel (A)}$ dual to coloop-free matroids, 
which is exactly the bounded part of $\Be{\stiefel (A)}$ if $\underm{V} = U_{d,n}$.

\subsection{Principal bundles}
The Stiefel map was given the name ``Stiefel'' to reflect the fact that
the space of tropical matrices maps to the space of valuated matroids
just as the \newword{non-compact Stiefel manifold} of $d\times n$ matrices of rank~$d$ 
maps to the Grassmannian of $d$-planes in $n$-space.

\Cref{thm_B} mirrors the classical fact that the non-compact Stiefel manifold
is a principal $\mathrm{GL}_d$ bundle over the Grassmannian, as we now explain.
The only invertible matrices of tropical numbers are the generalized permutation matrices,
those which have exactly one finite entry in every row and column, forming a group isomorphic to~$\mathbb R\wr S_d$.
\Cref{thm_B} implies that the space of $d\times n$ tropical matrices without too many infinities (\Cref{rem:domain})
has a deformation retract onto the Minkowski sum of the set of apices and the lineality space,
which is a ramified $\mathbb R\wr S_d$ bundle over its image.
The ramification arises because an apex can have equal rows.

It remains an open question to describe the topology of the image of the tropical Stiefel map.
The above bundle perspective suggests a possible approach.

\bibliography{Stiefel_presentations}

\begin{thebibliography}{10}

\bibitem{AGG}
Marianne Akian, Stephane Gaubert, and Alexander Guterman.
\newblock Tropical polyhedra are equivalent to mean payoff games.
\newblock {\em International Journal of Algebra and Computation},
  22(01):1250001, 2012.

\bibitem{TCDR}
Marianne Akian, St{\'e}phane Gaubert, and Alexander Guterman.
\newblock Tropical {C}ramer determinants revisited.
\newblock {\em Tropical and idempotent mathematics and applications},
  616:1--45, 2014.

\bibitem{ArdilaDevelin}
Federico Ardila and Mike Develin.
\newblock Tropical hyperplane arrangements and oriented matroids.
\newblock {\em Mathematische Zeitschrift}, 262(4):795--816, 2009.

\bibitem{AK06}
Federico Ardila and Caroline~J Klivans.
\newblock The {B}ergman complex of a matroid and phylogenetic trees.
\newblock {\em Journal of Combinatorial Theory, Series B}, 96(1):38--49, 2006.

\bibitem{BakerBowler}
Matthew Baker and Nathan Bowler.
\newblock Matroids over partial hyperstructures.
\newblock {\em Adv. Math.}, 343:821--863, 2019.

\bibitem{Bix77}
Robert~E Bixby.
\newblock A simple proof that every matroid is an intersection of fundamental
  transversal matroids.
\newblock {\em Discrete Mathematics}, 18(3):311--312, 1977.

\bibitem{BoninTransversalNotes}
Joseph~E Bonin.
\newblock An introduction to transversal matroids, 2010.

\bibitem{BrualdiMason}
Richard Brualdi and John Mason.
\newblock Transversal matroids and {H}all's theorem.
\newblock {\em Pacific Journal of Mathematics}, 41(3):601--613, 1972.

\bibitem{BrualdiWhite}
Richard~A Brualdi.
\newblock Transversal matroids.
\newblock In Neil White, editor, {\em Combinatorial geometries}, volume~29,
  pages 72--97. Cambridge Univ. Press, Cambridge, 1987.

\bibitem{BrualdiDinolt}
Richard~A Brualdi and George~W Dinolt.
\newblock Characterizations of transversal matroids and their presentations.
\newblock {\em Journal of Combinatorial Theory, Series B}, 12(3):268--286,
  1972.

\bibitem{Butkovich}
Peter Butkovi\v{c}.
\newblock {\em Max-linear systems: theory and algorithms}.
\newblock Springer Monographs in Mathematics. Springer-Verlag London, Ltd.,
  London, 2010.

\bibitem{CGM}
Colin Crowley, Noah Giansiracusa, and Joshua Mundinger.
\newblock A module-theoretic approach to matroids.
\newblock arXiv:1712.03440.

\bibitem{dLRS}
Jes{\'u}s~A De~Loera, J{\"o}rg Rambau, and Francisco Santos.
\newblock {\em Triangulations: Structures for algorithms and applications}.
\newblock Springer, 2010.

\bibitem{DF}
Harm Derksen and Alex Fink.
\newblock Valuative invariants for polymatroids.
\newblock {\em Advances in Mathematics}, 225(4):1840--1892, 2010.

\bibitem{DS}
Mike Develin and Bernd Sturmfels.
\newblock Tropical convexity.
\newblock {\em Doc. Math}, 9(1-27):7--8, 2004.

\bibitem{DressWenzel}
Andreas~WM Dress and Walter Wenzel.
\newblock Valuated matroids.
\newblock {\em Advances in Mathematics}, 93(2):214--250, 1992.

\bibitem{Edmonds1967}
Jack Edmonds.
\newblock Systems of distinct representatives and linear algebra.
\newblock {\em J. Res. Nat. Bur. Standards Sect. B}, 71B:241--245, 1967.

\bibitem{EdmondsFulkerson}
Jack Edmonds and Delbert~Ray Fulkerson.
\newblock Transversals and matroid partition.
\newblock Technical report, Rand Corp., Santa Monica, CA, 1965.

\bibitem{FS05}
Eva~Maria Feichtner and Bernd Sturmfels.
\newblock Matroid polytopes, nested sets and {B}ergman fans.
\newblock {\em Portugaliae Mathematica (new series)}, 62(4):437--468, 2005.

\bibitem{FR}
Alex Fink and Felipe Rinc{\'o}n.
\newblock Stiefel tropical linear spaces.
\newblock {\em Journal of Combinatorial Theory, Series A}, 135:291--331, 2015.

\bibitem{Frenk}
Bart Frenk.
\newblock {\em Tropical varieties, maps and gossip}.
\newblock PhD thesis, Technische Universiteit Eindhoven, 2013.

\bibitem{FultonSturmfels}
William Fulton and Bernd Sturmfels.
\newblock Intersection theory on toric varieties.
\newblock {\em Topology}, 36(2):335--353, 1997.

\bibitem{GK}
St{\'e}phane Gaubert and Ricardo~D Katz.
\newblock The {M}inkowski theorem for max-plus convex sets.
\newblock {\em Linear Algebra and its Applications}, 421(2-3):356--369, 2007.

\bibitem{HJS}
Sven Herrmann, Michael Joswig, and David~E Speyer.
\newblock Dressians, tropical {G}rassmannians, and their rays.
\newblock In {\em Forum Mathematicum}, volume 26, issue 6, pages 1853--1881. De
  Gruyter, 2014.

\bibitem{IngletonPiff}
A.~W. Ingleton and M.~J. Piff.
\newblock Gammoids and transversal matroids.
\newblock {\em J. Combinatorial Theory Ser. B}, 15:51--68, 1973.

\bibitem{Ingleton}
AW~Ingleton.
\newblock Transversal matroids and related structures.
\newblock In {\em Higher Combinatorics}, pages 117--131. Springer, 1977.

\bibitem{JoswigLoho}
Michael Joswig and Georg Loho.
\newblock Weighted digraphs and tropical cones.
\newblock {\em Linear Algebra and its Applications}, 501:304--343, 2016.

\bibitem{Kapranov1992}
M.~M. Kapranov, B.~Sturmfels, and A.~V. Zelevinsky.
\newblock Chow polytopes and general resultants.
\newblock {\em Duke Math. J.}, 67(1):189--218, 1992.

\bibitem{KobayashiMurota}
Yusuke Kobayashi and Kazuo Murota.
\newblock Induction of {M}-convex functions by linking systems.
\newblock {\em Discrete Appl. Math.}, 155(11):1471--1480, 2007.

\bibitem{MaclaganSturmfels}
Diane Maclagan and Bernd Sturmfels.
\newblock {\em Introduction to tropical geometry}, volume 161.
\newblock American Mathematical Soc., 2015.

\bibitem{MasonGammoids}
J.~H. Mason.
\newblock On a class of matroids arising from paths in graphs.
\newblock {\em Proc. London Math. Soc. (3)}, 25:55--74, 1972.

\bibitem{Mason}
JH~Mason.
\newblock A characterization of transversal independence spaces.
\newblock In {\em Th\'eorie des matro\"\i des}, pages 86--94. Springer, 1971.

\bibitem{Mitrofanov}
Mikhail Mitrofanov.
\newblock Bimatroids and {G}auss decomposition.
\newblock {\em European J. Combin.}, 28(4):1180--1195, 2007.

\bibitem{Mundinger}
Joshua Mundinger.
\newblock The image of a tropical linear space.
\newblock arXiv:1808.02150.

\bibitem{Murota-bimatroids}
K.~Murota.
\newblock Finding optimal minors of valuated bimatroids.
\newblock {\em Appl. Math. Lett.}, 8(4):37--41, 1995.

\bibitem{Murota}
Kazuo Murota.
\newblock {\em Matrices and matroids for systems analysis}, volume~20 of {\em
  Algorithms and Combinatorics}.
\newblock Springer-Verlag, Berlin, 2000.

\bibitem{Perfect}
Hazel Perfect.
\newblock Independence spaces and combinatorial problems.
\newblock {\em Proceedings of the London Mathematical Society}, 3(1):17--30,
  1969.

\bibitem{PostPAB}
Alexander Postnikov.
\newblock Permutohedra, associahedra, and beyond.
\newblock {\em Int. Math. Res. Not.}, 2009(6):1026--1106, 2009.

\bibitem{RI80}
Andr{\'a}s Recski and Masao Iri.
\newblock Network theory and transversal matroids.
\newblock {\em Discrete Applied Mathematics}, 2(4):311--326, 1980.

\bibitem{FirstSteps}
J\"{u}rgen Richter-Gebert, Bernd Sturmfels, and Thorsten Theobald.
\newblock First steps in tropical geometry.
\newblock In {\em Idempotent mathematics and mathematical physics}, volume 377
  of {\em Contemp. Math.}, pages 289--317. Amer. Math. Soc., Providence, RI,
  2005.

\bibitem{Rincon}
Felipe Rinc{\'o}n.
\newblock Local tropical linear spaces.
\newblock {\em Discrete \& Computational Geometry}, 50(3):700--713, 2013.

\bibitem{Sergeev}
Sergey Sergeev.
\newblock Multiorder, {K}leene stars and cyclic projectors in the geometry of
  max cones.
\newblock In G.L. Litvinov and S.N. Sergeev, editors, {\em Tropical and
  Idempotent Mathematics}, Contemporary mathematics. American Mathematical
  Society, 2009.

\bibitem{SS}
David Speyer and Bernd Sturmfels.
\newblock The tropical {G}rassmannian.
\newblock {\em Adv. Geom.}, 4(3):389--411, 2004.

\bibitem{SpeyerWilliams}
David Speyer and Lauren Williams.
\newblock The tropical totally positive {G}rassmannian.
\newblock {\em J. Algebraic Combin.}, 22(2):189--210, 2005.

\bibitem{Speyer}
David~E Speyer.
\newblock Tropical linear spaces.
\newblock {\em SIAM Journal on Discrete Mathematics}, 22(4):1527--1558, 2008.

\bibitem{YuYuster}
Josephine Yu and Debbie~S. Yuster.
\newblock Representing tropical linear spaces by circuits.
\newblock math/0611579.

\end{thebibliography}
\bibliographystyle{plain}

\end{document}